\documentclass[11pt,twoside]{amsart}

      \usepackage{amssymb}
      \usepackage{graphicx}
      \usepackage{dcolumn,indentfirst,cite,color,hyperref}

      \theoremstyle{plain}
     \newtheorem{thm}{Theorem}[section]
\newtheorem{lem}[thm]{Lemma}

\newtheorem{pro}[thm]{Proposition}
\newtheorem{rmk}[thm]{Remark}

\newcommand{\bess}{\begin{eqnarray*}}
\newcommand{\eess}{\end{eqnarray*}}
\input xy
\xyoption{all}

\usepackage{fancyhdr} 
\begin{document}

%




\author{Weiyuan Qiu}
\address{School of Mathematical Sciences, Fudan University, Shanghai, 200433, P.R.China}
\email{wyqiu@fudan.edu.cn}

\author{Pascale Roesch}
\address{Institut de Mathématiques de Toulouse, 118 route de Narbonne, F-31062, Toulouse Cedex 9, France}
\email{roesch@math.univ-toulouse.fr}

\author{Xiaoguang Wang}
\address{Department of Mathematics, Zhejiang University, Hangzhou, 310027, P.R.China}
\email{wxg688@163.com}

 \author{Yongcheng Yin}
 \address{Department of Mathematics, Zhejiang University, Hangzhou, 310027, P.R.China}
 \email{yin@zju.edu.cn}






   \title[]{Hyperbolic components of McMullen maps}


   \begin{abstract}
 In this article, we study the hyperbolic components of McMullen maps. We show that
  the boundaries of all hyperbolic components are  Jordan curves. This settles a problem posed by  Devaney. As a consequence,
   we  show that cusps are dense on the boundary of the unbounded hyperbolic component. This is a dynamical analogue of
  McMullen's theorem that  cusps  are dense on the Bers' boundary of Teichm\"uller space. 
   \end{abstract}

   \subjclass[2010]{Primary 37F45; Secondary 37F10, 37F15}

   \keywords{parameter plane, McMullen map, hyperbolic component, Jordan curve}



   \date{\today}


   \maketitle




\section{Introduction}\label{int}

The rational maps on the Riemann sphere $\mathbb{\widehat{C}}=\mathbb{C}\cup\{0\}$
$$z\mapsto z^d+\lambda z^{-m}, \lambda\in \mathbb{C}^*=\mathbb{C}\setminus \{0\},\ d,m\geq1$$
regarded as a singular perturbation of the monomial  $z\mapsto z^d$, take a  simple form but exhibit very rich dynamical behavior. These maps are known as `McMullen maps', since McMullen \cite{Mc1}  first studied these maps and pointed out that when $(d,m)=(2,3)$ and $\lambda$ is small,
the Julia set is a Cantor set of circles.
This family attracts many people for several reasons.
 The notable one is probably  that the Julia set varies in several classic fractals. It can be homeomorphic to either a  Cantor set, or a  Cantor set of circles, or a Sierpinski carpet \cite{DLU}.
Another reason is that this family provides many examples for different purpose  to understand the dynamic of rational maps. We refer the reader to
\cite{D1,D2,D3,DK,DLU,DP,HP,WQY,R1,S}
 and the reference therein for a
number of related results.

The purpose of this article is to study the boundaries of the hyperbolic components of  the   McMullen maps:
$$f_{\lambda}:z\mapsto z^n+\lambda z^{-n}, \ \ \lambda\in \mathbb{C}^*,\ n\geq3.$$
For any
$\lambda\in\mathbb{C}^*$, the map $f_\lambda$ has a superattracting
fixed point at $\infty$. The immediate attracting basin of $\infty$ is denoted
by $B_\lambda$.
The critical set  of $f_\lambda$ is $\{0,\infty\} \cup C_\lambda$, where
$C_\lambda=\{c\in\mathbb{C}; c^{2n}=\lambda\}$. Besides
$\infty$, there are only two critical values:
$v_\lambda^+=2\sqrt{\lambda}$ and $v_\lambda^-=-2\sqrt{\lambda}$ (here, when restricted to the fundamental domain,
$v_\lambda^+$ and $v_\lambda^-$ are well-defined, see Section \ref{cut}).
In
fact, there is only one free critical orbit (up to a sign).

 Recall that a rational map is {\it hyperbolic} if all critical orbits are attracted by the attracting cycles, see \cite{M, Mc4}.
 A McMullen map $f_\lambda$ is hyperbolic if the free critical orbit is attracted either by $\infty$ or by an attracting cycle in $\mathbb{C}$.
  It is known (see Theorem \ref{1a}) that for the family $\{f_\lambda\}_{\lambda\in\mathbb{C}^*}$, every hyperbolic component is isomorphic to either the unit disk $\mathbb{D}$ or
 $\mathbb{D}^*=\mathbb{D}-\{0\}$. This family admits the `Yoccoz puzzle' structure (see  \cite{WQY}), which allows us to carry out further study of their dynamical
behavior and the boundaries of the hyperbolic components. The Yoccoz puzzle is induced by a kind of Jordan curve called `cut ray' which was first constructed by Devaney \cite{D3}.

The main result of the paper is:
\begin{thm}\label{main} The boundaries of all hyperbolic components are  Jordan curves.
\end{thm}

 Theorem \ref{main} affirmly answers a problem posed by
Devaney \cite{DK} at  the Snowbird Conference on the 25th Anniversary of the Mandelbrot set. In fact, Devaney has proven in \cite{D1} that the boundary of the
hyperbolic component containing the punctured neighborhood of the origin is a Jordan curve and he  asked whether all the other  hyperbolic components of {\it escape type} (the free critical orbit escapes to $\infty$) are Jordan domains. Our result confirms this.

\begin{figure}[!htpb]
  \setlength{\unitlength}{1mm}
  {\centering
  \includegraphics[width=60mm]{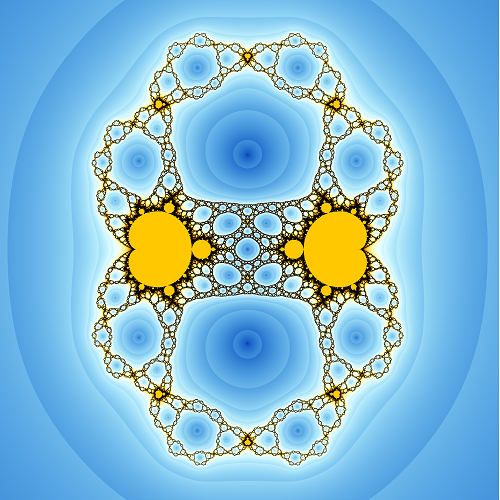}
  \includegraphics[width=60mm]{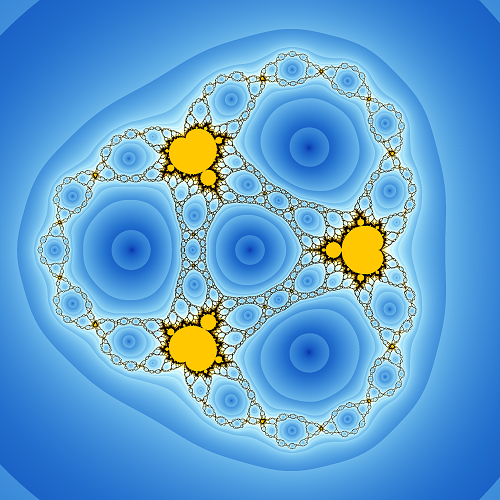}\put(-8,28){$\mathcal{H}_0$}\put(-47,30){$\mathcal{H}_3$}
  \put(-25,16){$\mathcal{H}_3$}\put(-32,28){$\mathcal{H}_2$}\put(-25,40){$\mathcal{H}_3$}
  \put(-95,29){$\mathcal{H}_2$}\put(-70,30){$\mathcal{H}_0$}\put(-94,16){$\mathcal{H}_3$}\put(-94,40){$\mathcal{H}_3$}}
\begin{minipage}{14cm}
  \caption{Parameter plane of McMullen maps, $n=3,4$.}
  \label{Fig_McMullen}
  \end{minipage}
\end{figure}
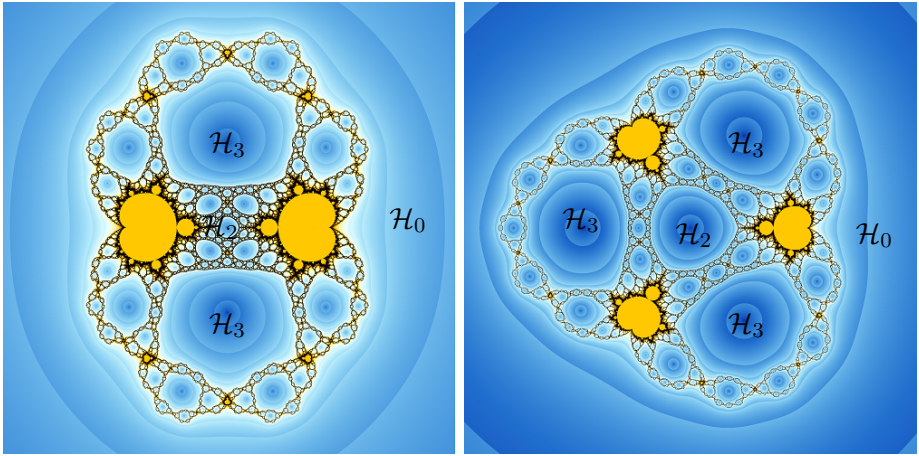



Our second result concerns  the  topological structure of the boundary of the unbounded hyperbolic component $\mathcal{H}_0$ (consisting
 of the parameters  for which the Julia set $J(f_\lambda)$ is a Cantor set, see Section \ref{escape}):


\begin{thm}\label{cusp} Cusps are dense in $\partial\mathcal{H}_0$.
\end{thm}

Here, according to McMullen \cite{Mc3}, a parameter $\lambda\in \mathbb{C}^*$ is called a {\it cusp} if the map $f_\lambda$ has a parabolic cycle on $\partial B_\lambda$.
Theorem \ref{cusp} is a dynamical analogue of McMullen's theorem that cusps are dense on the  Bers' boundary of Teichm\"uller space \cite{Mc2}.

 Let's sketch how to obtain Theorem \ref{cusp}.
Assuming Theorem \ref{main}, one gets a canonical parameterization $\nu:\mathbb{S}\rightarrow \partial \mathcal{H}_0$, where
$\nu(\theta)$ is defined to be the landing point of the parameter ray $\mathcal{R}_0(\theta)$  (see Section \ref{0Jordan}) in $\mathcal{H}_0$.
We actually give a complete characterization of $\partial \mathcal{H}_0$ and its cusps:

\begin{thm}[Characterization of $\partial \mathcal{H}_0$ and cusps]\label{cc}

\

1.  $\lambda\in \partial \mathcal{H}_0$ if and
only if $\partial B_\lambda$ contains either 
$C_\lambda$ or a parabolic cycle. 

2.  $\nu(\theta)$ is a cusp if and only if $n^p\theta\equiv\theta {\ \rm mod \ } \mathbb{Z}$ for some $p\geq1$.
\end{thm}

 Theorem \ref{cusp} is an immediate consequence of Theorem \ref{cc} since  $\{\theta; n^p\theta\equiv\theta {\ \rm mod \ } \mathbb{Z}, p\geq1\}$ is
a dense subset of the unit circle $\mathbb{S}$.

The heart part of paper is to prove Theorem  \ref{main}. We briefly sketch the idea of the proof and the organization of the paper.
The idea  is  different from the {\it parapuzzle} techniques (known to be a  powerful  tool to study the boundary of hyperbolic components, see \cite{R2,R3}).
 We rely more on the dynamical {\it Yoccoz puzzle} rather than the parapuzzle.
From Section \ref{escape} to Section \ref{sier},  we study the  hyperbolic components of escape type, which are called {\it escape domains}.

In Section \ref{escape}, we  parameterize  the escape domains.

In Section \ref{cut}, we sketch the construction of the cut rays which are important in the study of escape domains. The crucial fact of  cut rays is that they move continuously in Hausdorff topology with respect to the parameter.

In Section \ref{0Jordan}, we will prove that $\partial\mathcal{H}_0$ is a Jordan curve.
We first show that $\partial\mathcal{H}_0$ is locally connected. To this end, we show that any two maps (which are not cusps) in the same impression of the parameter ray are quasiconformally conjugate (Proposition \ref{3ff}).
 This conjugacy is constructed with the help of cut rays  and it is holomorphic in the Fatou set.  A `zero measure argument' for non-renormalizable map following Lyubich (see Section \ref{leb}) implies that the conjugacy is actually a M\"obius map. So the two maps are  the same.
 After we knowing the local connectivity, a dynamical result (Theorem \ref{11a}) enables us to show that the boundary is a  Jordan curve.

In Section \ref{sier}, we will prove that the boundaries of all escape domains of level $k\geq3$ (these escape domains  are called {\it Sierpinski holes}) are Jordan curves. The proof is based on three
ingredients:  the boundary regularity of $\partial\mathcal{H}_0$, holomorphic motion and continuity of cut rays.  We remark that our approach also applies to $\partial\mathcal{H}_2$.  This will yield a different proof from Devaney's in \cite{D1}.

 In Section \ref{babe}, we show that  the hyperbolic components which are not of `escape type' are  Jordan domains.


\vskip0.2cm
\noindent {\bf Acknowledgement.} X. Wang would like to thank  the Institute for Computational and Experimental Research in Mathematics (ICERM)  for
hospitality and financial support. We would like to thank Xavier Buff for helpful discussions.

\section{Escape domains  and parameterizations}\label{escape}

There are two kinds of hyperbolic McMullen maps based on the behavior of the free critical orbit.  If the   orbit
escapes to infinity, the corresponding hyperbolic component is called an {\it escape domain}.  If the  orbit
tends to an attracting cycle other than $\infty$, the corresponding hyperbolic component is of  {\it renormalizable type}.

In this section, we present some known facts about escape domains.  We refer the reader to \cite{DLU} for more background materials.
The hyperbolic components  of renormalizable type will be discussed in Section \ref{babe}.

For any $\lambda\in \mathbb{C}^*$,  the {\it Julia set} $J(f_\lambda)$ of $f_\lambda$ can be identified as the boundary of $\cup_{k\geq0}f_\lambda^{-k}(B_\lambda)$. It satisfies
 $e^{\pi i/n}J(f_\lambda)=J(f_\lambda)$. The {\it Fatou set} $F(f_\lambda)$ of $f_\lambda$ is defined by   $F(f_\lambda)=\mathbb{\widehat{C}}- J(f_\lambda)$.
 We denote by $T_\lambda$ the component of $f_\lambda^{-1}(B_\lambda)$ containing $0$. It is possible that $B_\lambda=T_\lambda$. In that case, the critical set $C_\lambda\subset B_\lambda$ and $J(f_\lambda)$ is a Cantor set (see Theorem \ref{1c}).


For any $k\geq0$, we define a parameter set $\mathcal{H}_k$ as follows:
$$\mathcal{H}_k=\{\lambda\in
\mathbb{C}^*; k \text{ is the first integer such that } f_\lambda^{k}(C_\lambda)\subset B_\lambda\}.$$
A component of $\mathcal{H}_k$ is called a {\it escape domain} of level $k$.
One may verify that  $\mathcal{H}_0=\{\lambda\in
\mathbb{C}^*; v_\lambda^+\in B_\lambda\}, \mathcal{H}_1=\emptyset$ and $\mathcal{H}_k=\{\lambda\in
\mathbb{C}^*; f_\lambda^{k-2}(v_\lambda^+)\in T_\lambda\neq B_\lambda\}$  for $k\geq2$.
See Figure 1.
The complement of the escape domains is called the
{\it non-escape locus} $\mathcal{M}$. It can be written as
$$\mathcal{M}=\{\lambda\in\mathbb{C}^*; \  f_\lambda^k(v^+_\lambda)\text{ does not  tend to infinity  as }
 k\rightarrow \infty\}.$$
The set $\mathcal{M}$ is invariant under the maps $z\mapsto \overline{z}$ and $z\mapsto e^{\frac{2\pi i}{n-1}} z$.

 \begin{thm}[Escape
Trichotomy \cite{DLU} and Connectivity \cite{DR}]\label{1c}

\

1. If $\lambda\in \mathcal{H}_0$, then $J(f_\lambda)$ is a Cantor
set.

2. If $\lambda\in \mathcal{H}_2$, then $J(f_\lambda)$
is a Cantor set of circles.

3. If $\lambda\in \mathcal{H}_k$ for some $k\geq3$, then $J(f_\lambda)$ is a Sierpi\'nski curve.

4. If $\lambda\in \mathcal{M}$, then Julia set $J(f_\lambda)$ is connected.
\end{thm}

\begin{figure}[!htpb]
  \setlength{\unitlength}{1mm}
  {\centering
  \includegraphics[width=39mm]{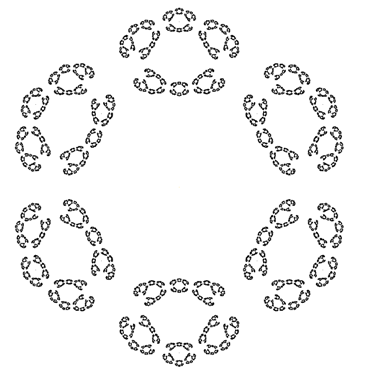}
  \includegraphics[width=41mm]{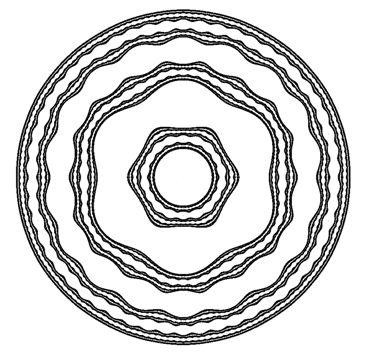}}\\
  {\centering\includegraphics[width=40mm]{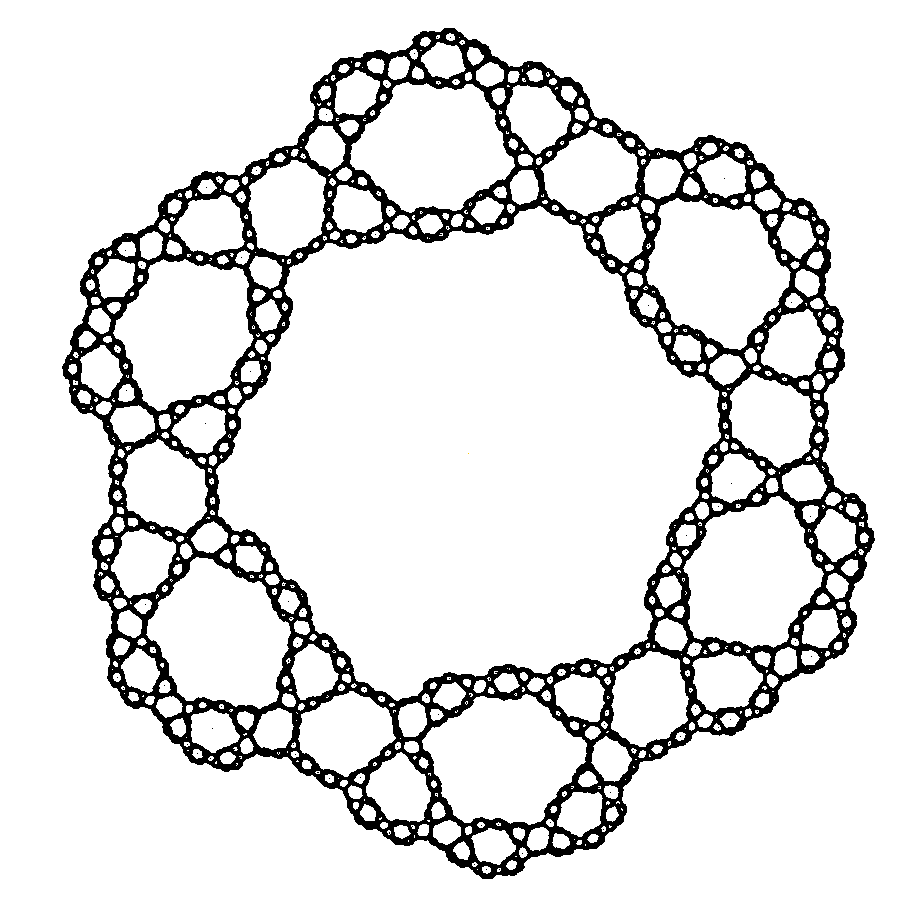}
  \includegraphics[width=42mm]{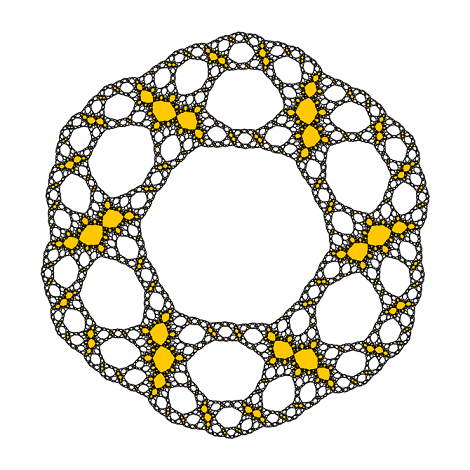}}
\begin{minipage}{14cm}
  \caption{The Julia sets: a Cantor set (upper-left),  a Cantor set of circles
  (upper-right), a Sierpinski curve (lower-left) and  a connected set (lower-right).}
  \label{Fig_McMullen}
  \end{minipage}
\end{figure}

Based on Theorem \ref{1c}, we give some remarks on the escape domains. According to Devaney, $\mathcal{H}_0$ is called the
 {\it Cantor set locus}, $\mathcal{H}_2$   is called the
{\it McMullen domain}, $\mathcal{H}_k$ with $k\geq3$  is called the
{\it Sierpinski locus} and each of its component is called a {\it Sierpinski hole}. Devaney  showed that the boundary $\partial \mathcal{H}_2$ is a Jordan curve \cite{D1}  and $\mathcal{H}_k$ with $k\geq3$  consists
of $(2n)^{k-3}(n-1)$ disk components \cite{D2}.

The B\"{o}ttcher map $\phi_\lambda$ of $f_\lambda$ is defined in a
neighborhood of $\infty$ by
$\phi_\lambda(z)=\displaystyle\lim_{k\rightarrow\infty}(f^k_\lambda(z))^{n^{-k}}.$
It is unique if we require $\phi'_\lambda(\infty)=1$. The map $\phi_\lambda$ satisfies
$\phi_\lambda(f_\lambda(z))=\phi_\lambda(z)^n$ and
 $\phi_\lambda(e^{\pi i/n} z)=e^{\pi i/n}\phi_\lambda(z)$.
   One may verify that near infinity,
$$\phi_\lambda(z)=\sum_{k\geq0}a_k(\lambda)z^{1-2kn}, \ a_0(\lambda)=1, a_1(\lambda)=\lambda/n, \cdots.$$
 If $\lambda\in \mathbb{C}^*\setminus \mathcal{H}_0$, then both $B_\lambda$ and $T_\lambda$ are simply connected. In that case,
there is  a unique Riemann mapping $\psi_{\lambda}:
T_\lambda\rightarrow\mathbb{D}$, such that
$\psi_{\lambda}(w)^{-n}=\phi_{\lambda}(f_\lambda(w))$ for $w\in T_\lambda$ and $\psi'_{\lambda}(0)=\sqrt[n]{\lambda}$.
The {\it external ray} $R_\lambda(t)$ of angle $t$ in $B_\lambda$ is defined by ${R}_\lambda(t):=\phi_\lambda^{-1}((1,+\infty)e^{2\pi i t})$, the
{\it internal ray} $R_{T_\lambda}(t)$ of angle $t$ in $T_\lambda$ is defined by $R_{T_\lambda}(t):=\psi_\lambda^{-1}((0,1)e^{2\pi i t})$.



\begin{thm}[{Parameterization of escape domains,\cite{D2}\cite{R1}\cite{S}}]\label{1a}

\

 1. $\mathcal{H}_0$ is the
unbounded component of $\mathbb{C}^*-\mathcal{M}$. The map
$\Phi_0:\mathcal{H}_0\rightarrow
\mathbb{C}-\overline{\mathbb{D}} $ defined by
$\Phi_0(\lambda)=\phi_\lambda(v_\lambda^+)^2$
is a conformal isomorphism.

2. $\mathcal{H}_2$ is the component of
$\mathbb{C}^*-\mathcal{M}$ containing the punctured neighborhood of
$0$. The holomorphic map $\Phi_2:\mathcal{H}_2\rightarrow
\mathbb{C}-\overline{\mathbb{D}} $ defined via
$\Phi_2(\lambda)^{n-2}={\phi_\lambda(f_\lambda(v_\lambda^+))^{2}}$ and
$\lim_{\lambda\rightarrow0}\lambda\Phi_2(\lambda)=2^{\frac{2n}{2-n}}$, is a conformal isomorphism .

3. Let $\mathcal{H}$ be a escape domain of level $k\geq3$.
The map   $\Phi_\mathcal{H}:\mathcal{H}\rightarrow{\mathbb{D}} $ defined by
$\Phi_\mathcal{H}(\lambda)=\psi_\lambda(f_\lambda^{k-2}(v_\lambda^+))$
is a conformal isomorphism.
\end{thm}

Both $\Phi_0$ and $\Phi_2$ satisfy $\Phi_\epsilon(e^{\frac{2\pi i}{n-1}}\lambda)=
e^{\frac{2\pi i}{n-1}}\Phi_\epsilon(\lambda)$ and  $\Phi_\epsilon(\overline{\lambda})=
\overline{\Phi_\epsilon(\lambda)}$ for $\epsilon\in\{0,2\}$ and $\lambda\in \mathcal{H}_\epsilon$.
Thus they take the forms $\Phi_\epsilon(\lambda)=\lambda\Psi_\epsilon(\lambda^{n-1})$, where $\Psi_\epsilon$ is holomorphic function whose expansion has real coefficients.

\begin{thm}[{Connectivity  of $\mathcal{M}$}]\label{1b} The non-escape locus $\mathcal{M}$ is connected and has logarithmic
capacity equal to 1/4.
\end{thm}
\begin{proof} By Theorem \ref{1a},  each  component of $\mathbb{\widehat{C}}-\mathcal{M}$ is a topological disk. So $\mathcal{M}$ is connected.  The   logarithmic
capacity of  $\mathcal{M}$  follows from the expansion of $\Phi_0$ near $\infty$:
 $\Phi_0(\lambda)=4\lambda+\mathcal{O}(\lambda^{2-n})$.
\end{proof}


%


\section{Cut rays in  the dynamical plane}\label{cut}

The topology of $\partial B_\lambda$ is considered  in \cite{WQY}, where the authors showed
\begin{thm}[\cite{WQY}]\label{11a}
For any $n\geq3$ and any $\lambda\in\mathbb{C}^*$,

$\bullet$
$\partial B_\lambda$ is either a Cantor set or a Jordan curve. In the latter case, all Fatou components eventually mapped to
  $B_\lambda$ are Jordan domains.

$\bullet$ If
$\partial B_\lambda$ is a Jordan curve  containing neither
 a parabolic point nor the recurrent critical set $C_\lambda$, then $\partial B_\lambda$ is a quasi-circle.
\end{thm}

Here, the critical set $C_\lambda$ is called {\it recurrent} if $C_\lambda\subset J(f_\lambda)$ and the set $\cup_{k\geq1}f^k_\lambda(C_\lambda)$ has an accumulation  point in $C_\lambda$.
The proof of Theorem \ref{11a} is based on  the Yoccoz puzzle  theory.
 To apply this theory, we need to construct
  a  kind of Jordan curve  which cuts the
Julia set into two connected parts. These curves are called {\it cut rays}.
They play a crucial role in our study of the boundaries of escape domains. For this, we briefly sketch their constructions here.

To begin, we identify the unit circle
$\mathbb{S}=\mathbb{R}/\mathbb{Z}$ with $(0,1]$. We define a map $\tau:  \mathbb{S}\rightarrow\mathbb{S}$ by $\tau
(\theta)=n\theta  \text{ mod 1} $. Let
$\Theta_k=(\frac{k}{2n},\frac{k+1}{2n}]$ for $0\leq k \leq n$ and
$\Theta_{-k}=(\frac{k}{2n}+\frac{1}{2},\frac{k+1}{2n}+\frac{1}{2}]$ for $1\leq k\leq n-1$. Obviously,
$(0,1]=\cup_{-n<j\leq n}\Theta_j$.


 Let $\Theta$ be the set of all angles $\theta\in (0,1]$ whose orbits
 remain in $\bigcup_{k=1}^{n-1}(\Theta_k\cup
\Theta_{-k})$ under all iterations of $\tau$.  One may
verify that $\Theta$ is a Cantor set. Given an angle $\theta\in \Theta$, the {\it itinerary} of $\theta$ is a sequence of symbols
 $(s_0,s_1,s_2,\cdots)\in
\{\pm1,\cdots,\pm(n-1)\}^{\mathbb{N}}$ such that $\tau^k(\theta)\in \Theta_{s_k}$ for all
$k\geq0$. The angle $\theta\in \Theta$ and its itinerary $(s_0,s_1,s_2,\cdots)$ satisfy the identity (\cite{WQY}, Lemma 3.1):

$$\theta=\frac{1}{2}\bigg(\frac{\chi(s_0)}{n}+\sum_{k\geq1}\frac{|s_k|}{n^{k+1}}\bigg),$$
where $\chi(s_0)=s_0$ if $0\leq s_0\leq n$ and  $\chi(s_0)=n-s_0$  if $-(n-1)\leq s_0\leq -1$.

Note that  ${e}^{\pi i/(n-1)}f_\lambda
(z)=(-1)^nf_{{e}^{2\pi i/(n-1)}\lambda}({e}^{\pi i/(n-1)}z)$ for all  $\lambda\in \mathbb{C}^*$. This implies that
the fundamental domain  of the parameter plane is
$$\mathcal{F}_0=\{\lambda\in \mathbb{C^*}; 0\leq\arg\lambda<{2\pi}/(n-1)\}.$$ We denote the  interior
of $\mathcal{F}_0$ by
 $$\mathcal{F}:=\{\lambda\in \mathbb{C^*}; 0<\arg\lambda<{2\pi}/(n-1)\}.$$

In our discussion,  we assume $\lambda\in \mathcal{F}_0$ and
let $O_\lambda=\cup_{k\geq0}f_\lambda^{-k}(\infty)$ be the grand orbit of $\infty$.
Let $c_0=c_0(\lambda)=\sqrt[2n]{\lambda}$ be the critical point that
lies on $\mathbb{R}^+:=[0, +\infty)$ when $\lambda\in \mathbb{R}^+$ and varies
analytically as $\lambda$ ranges over $\mathcal{F}$. Let
$c_k(\lambda)=c_0{e}^{k\pi i/n}$ for $1\leq k\leq 2n-1$. The critical points
$c_k$ with $k$ even are mapped to $v_\lambda^+$
while the critical points $c_k$ with $k$ odd are mapped to
$v_\lambda^-$.

Let $\ell_k=c_k[0,+\infty]$ be the closed straight line connecting $0$ to $\infty$ and passing
through $c_k$ for $0\leq k\leq 2n-1$. The closed sector bounded by $\ell_k$ and
$\ell_{k+1}$ is denoted by $S_k^\lambda$ for $0\leq k\leq n$. Define
$S_{-k}^\lambda=-S_k^\lambda$ for $1\leq k\leq n-1$. These sectors are
arranged counterclockwise about the origin as
$S_0^\lambda,S_1^\lambda,\cdots,S_n^\lambda,S_{-1}^\lambda,\cdots,S_{-(n-1)}^\lambda$.

The critical value $v_\lambda^+$ always lies in $S_0^\lambda$ because $\arg
c_0\leq\arg v_\lambda^+<\arg c_1$ for all $\lambda\in\mathcal{F}_0$.
Correspondingly, the critical value $v_\lambda^-$ lies in $S_n^\lambda$.
The image of $\ell_k$ under $f_\lambda$ is a
straight ray connecting one of the critical values to $\infty$; this
ray is called a \textit{critical value ray}. As a consequence,
$f_\lambda$ maps the interior of each of the sectors of
${S_{\pm1}^\lambda,\cdots,S_{\pm(n-1)}^\lambda}$ univalently onto a region
$\Upsilon_\lambda$, which can be identified as the complex sphere
$\mathbb{\widehat{C}}$ minus two critical value rays. For any $\epsilon\in \{\pm1,\cdots,\pm(n-1)\}$, let $int(S_\epsilon^\lambda)$ be the interior of $S_\epsilon^\lambda$,  the inverse of
$f_\lambda:int(S_\epsilon^\lambda)\rightarrow \Upsilon_\lambda$  is denoted by $h_\epsilon^\lambda:\Upsilon_\lambda\rightarrow int(S_\epsilon^\lambda)$.

\begin{thm}[Cut ray, \cite{D3} \cite{WQY}]\label{2a} For  any $\lambda\in\mathcal{F}$ and any  angle  $\theta\in\Theta$ with itinerary
 $(s_0,s_1,s_2,\cdots)$, the
set $$\Omega_\lambda^\theta:=\bigcap_{k\geq0} f_\lambda^{-k}(S_{s_k}^\lambda\cup S_{-s_k}^\lambda)$$
is a Jordan curve intersecting the Julia set $J(f_\lambda)$ in
a Cantor set. 
\end{thm}

Theorem \ref{2a} is originally proven for  the parameters $\lambda\in \mathcal{F}\cap\mathcal{M}$ in \cite{WQY}. The proof actually works for all $\lambda\in\mathcal{F}$
without any difference.

Here are some facts about the cut rays:
 $\Omega_\lambda^\theta=-\Omega_\lambda^\theta$ and
 $\Omega_\lambda^\theta=\Omega_\lambda^{\theta+1/2}$;  $R_\lambda(\theta)\cup R_\lambda(\theta+\frac{1}{2})\subset\Omega_\lambda^\theta\cap F(f_\lambda)\subset \cup_{k\geq0}f_\lambda^{-k}(B_\lambda)$;
 $0,\infty\in \Omega_\lambda^\theta$ and $\Omega_\lambda^\theta\setminus \{0,\infty\}$ is
contained in the interior of $S_{s_0}^\lambda\cup S_{-s_0}^\lambda$;  $f_\lambda(\Omega^\theta_\lambda)=\Omega_\lambda^{\tau(\theta)}$ and
$f_\lambda:\Omega_\lambda^{\theta}
 \rightarrow\Omega_\lambda^{\tau(\theta)}$
is a two-to-one map. We refer the reader to
\cite{WQY} for more details of the cut rays.

\begin{figure}[h]
\begin{center}
\includegraphics[height=6cm]{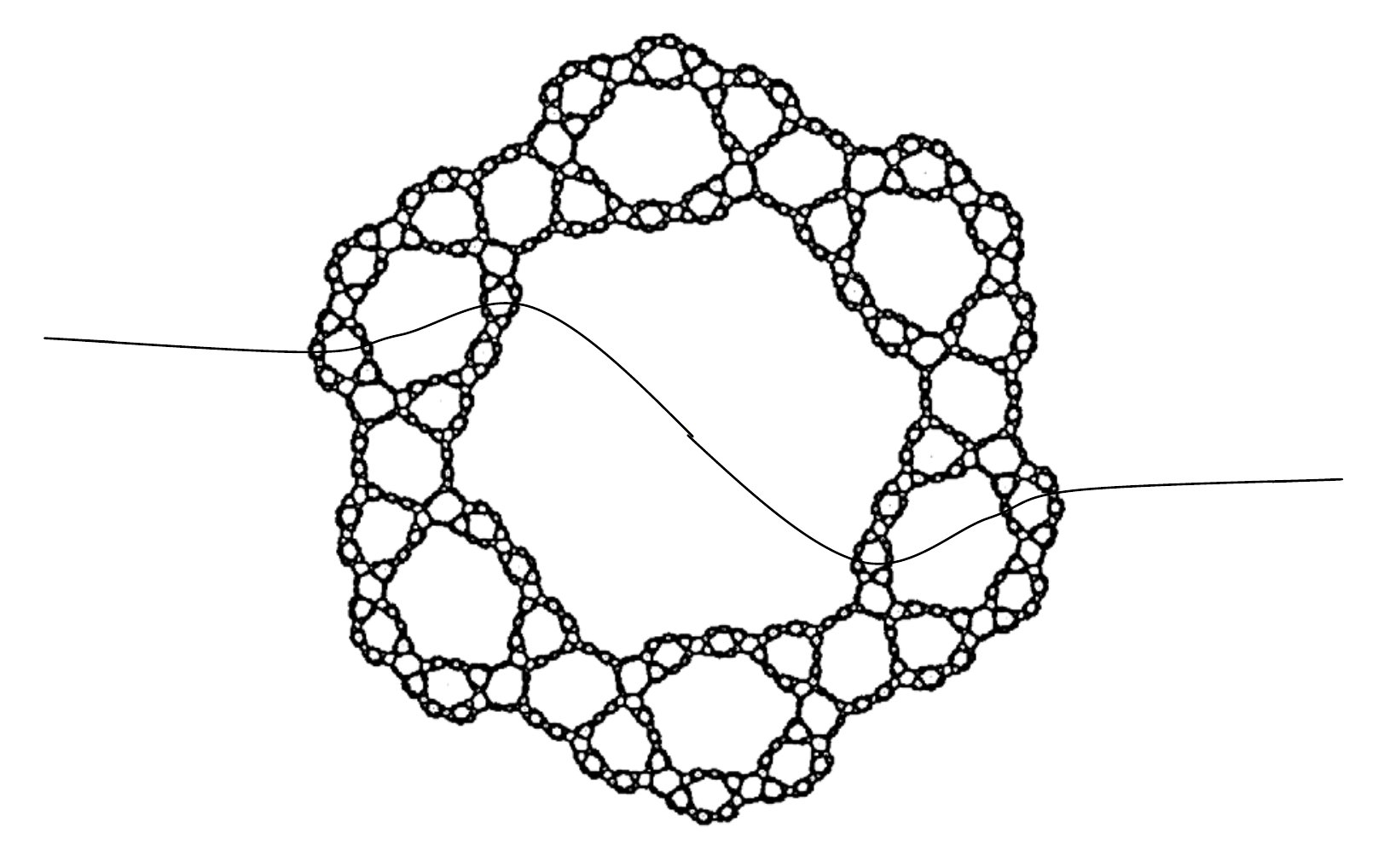}
\put(-140,73){$0$} \put(-40,83){$R_\lambda(1)$} \put(-270,85){$R_\lambda(1/2)$}
 \caption{An example of cut ray: $\Omega_\lambda^{1}=\Omega_\lambda^{1/2}$. $(n=3)$}
\end{center}\label{f5}
\end{figure}

Now we give some new dynamical properties of the cut rays.
These facts are useful to study the parameter plane. We denote by $B(z,r)$ the Euclidean disk centered at $z$ with radius $r$. For any
$\lambda\in \mathbb{C}^*\setminus\mathcal{H}_0$, set $B_\lambda^L:=\{w\in B_\lambda;
|\phi_\lambda(w)|>L\}$ for $L\geq1$.

\begin{lem}[Holomorphic motion of the cut rays] \label{3e}  Fix an angle $\theta\in \Theta$,  the cut ray $\Omega_\lambda^\theta$ moves holomorphically with respect to  $\lambda\in \mathcal{F}$.

\end{lem}

\begin{proof} Fix a parameter $\lambda_0\in \mathcal{F}$.  We will define a holomorphic motion $h: \mathcal{F}\times
((\Omega_{\lambda_0}^\theta\setminus O_{\lambda_0}) \cap
F(f_{\lambda_0}))\rightarrow \mathbb{\widehat{C}}$ with base point
$\lambda_0$  as follows. For any $\lambda\in\mathcal{F}$, there is  a number $L\geq1$ (depending on $\lambda$) such
that  the B\"ottcher map
$\phi_u: B_u^L\rightarrow \{\zeta\in \mathbb{\widehat{C}};
|\zeta|>L\}$ is a conformal isomorphism, for $u\in\{\lambda,\lambda_0\}$.

 If $z\in (\Omega_{\lambda_0}^\theta\setminus O_{\lambda_0})\cap
B_{\lambda_0}^L$, we define
$h(\lambda,z)=\phi_\lambda^{-1}\circ\phi_{\lambda_0}(z)$. If $z\in
(\Omega_{\lambda_0}^\theta\setminus O_{\lambda_0}) \cap
(F(f_{\lambda_0})\setminus B_{\lambda_0}^L)$, we consider the
{\it itinerary} of $z$, which is the unique sequence of symbols $(\epsilon_0,\epsilon_1,\epsilon_2,\cdots)\in \{\pm1,\cdots,\pm(n-1)\}^{\mathbb{N}}$
such that $f_{\lambda_0}^k(z)\in S_{\epsilon_k}^{\lambda_0}$ for all $k\geq0$.
 Let $N\geq1$ be the
first integer such that $f_{\lambda_0}^N(z)\in B_{\lambda_0}^L$. We
define $h(\lambda,z)=h^\lambda_{\epsilon_0}\circ\cdots\circ
h^\lambda_{\epsilon_{N-1}}
\circ\phi_{\lambda}^{-1}\circ\phi_{\lambda_0}(f_{\lambda_0}^N(z))$.
In this way, we  get a well-defined map $h: \mathcal{F}\times
((\Omega_{\lambda_0}^\theta\setminus O_{\lambda_0}) \cap
F(f_{\lambda_0}))\rightarrow \mathbb{\widehat{C}}$. Since both
$\phi_\lambda$ and $h^\lambda_{\epsilon_j}$ are holomorphic with
respect to $\lambda\in \mathcal{F}$, one may verify that the map $h$
is a holomorphic motion parameterized by $\mathcal{F}$, with base
point $\lambda_0$ (namely, $h(\lambda_0,z)\equiv z$). Moreover, for any $\lambda\in \mathcal{F}$, we
have $h(\lambda, (\Omega_{\lambda_0}^\theta\setminus O_{\lambda_0})
\cap F(f_{\lambda_0}))=(\Omega_{\lambda}^\theta\setminus
O_{\lambda}) \cap F(f_{\lambda})$.

Note that for any $\lambda\in \mathcal{F}$, the closure of
$(\Omega_{\lambda}^\theta\setminus O_{\lambda}) \cap
F(f_{\lambda})$ is $\Omega_{\lambda}^\theta$.
 By the $\lambda$-Lemma (see \cite{MSS} or \cite{Mc4}), there is a holomorphic
motion $H: \mathcal{F}\times \Omega_{\lambda_0}^\theta\rightarrow
\mathbb{\widehat{C}}$ extending $h$ and for any $\lambda\in
\mathcal{F}$, one has  $H(\lambda,
\Omega_{\lambda_0}^\theta)=\Omega_{\lambda}^\theta$. That is to say,
the cut ray $\Omega_{\lambda}^\theta$ moves holomorphically  when $\lambda$ ranges over
$\mathcal{F}$.
\end{proof}

The following result will be used to prove Proposition \ref{3ff}.

\begin{lem}[Periodic cut rays are quasi-circles] \label{3ef} For any $\lambda\in\mathcal{F}$ and any periodic angle $\theta\in \Theta$,
the cut ray $\Omega_\lambda^\theta$ is a quasi-circle.
\end{lem}

It's not clear whether  $\Omega_\lambda^\theta$ is  a quasi-circle when $\theta\in\Theta$ is not rational. But Lemma \ref{3ef} suffices for our purposes.

\begin{proof} Let $p>0$ be the first integer such
$f_\lambda^p(\Omega_\lambda^\theta)=\Omega_\lambda^\theta$. We fix
some large number $L>1$ so that  $\phi_\lambda: B_\lambda^L\rightarrow \{\zeta\in \mathbb{\widehat{C}};
|\zeta|>L\}$ is a conformal map.

Since the cut rays  $\Omega_\lambda^\theta,\cdots,
\Omega_\lambda^{\tau^{p-1}(\theta)}$ avoid the free critical values $v_\lambda^{\pm}$, there is a number
 $\delta_0>0$ such that for any $z\in \Omega_\lambda^\theta\setminus
\{\infty\}$, any integer $k\geq0$ and any component $U_k(z)$ of
$f_\lambda^{-k}(B(z,\delta_0))$ intersecting with
$\Omega_\lambda^\theta\cup\cdots\cup
\Omega_\lambda^{\tau^{p-1}(\theta)}$, we have that $U_k(z)$ is a disk and $f_\lambda^k:{U_k(z)}\rightarrow B(z,\delta_0)$ is a conformal map.

Before further discussion, we need a  fact.

\textbf{Fact:} {\it Let $\gamma$ be a
Jordan curve in  $\mathbb{\widehat{C}}$, then for any $\epsilon>0$,  there is a constant $\delta_\epsilon>0$ such that if
$z_1,z_2\in\gamma$ satisfying $d_{\mathbb{\widehat{C}}}(z_1,z_2)<\delta_\epsilon$, then
$\min\{{\rm diam}_{\mathbb{\widehat{C}}}(\gamma_1),{\rm diam}_{\mathbb{\widehat{C}}}(\gamma_2)\}$
$<\epsilon$, where
$\gamma_1, \gamma_2$ are two components of $\gamma-\{z_1,z_2\}$,
$d_{\mathbb{\widehat{C}}}$ is the  spherical distance  and ${\rm diam}_{\mathbb{\widehat{C}}}$ is the spherical diameter.}


Note that the Euclidean distance  is comparable with the spherical distance in any compact subset of $\mathbb{C}$.
Since
$\Omega_\lambda^\theta,\cdots,\Omega_\lambda^{\tau^{p-1}(\theta)}$
are Jordan curves on $\mathbb{\widehat{C}}$, it follows from the
above fact that for  any $\epsilon>0$,  there is a number
$\delta(\epsilon)>0$ so that for any $0\leq j<p$ and any pair
$\zeta_1,\zeta_2\in \Omega_\lambda^{\tau^{j}(\theta)}\setminus
f_\lambda^{-1}(B_\lambda^{2L})$, the condition $|\zeta_1-\zeta_2|<\delta(\epsilon)$
implies the Euclidean diameter ${\rm diam }(L(\zeta_1,\zeta_2))<\epsilon$, where $L(\zeta_1,\zeta_2)$ is the bounded component
 of $\Omega_\lambda^{\tau^{j}(\theta)}\setminus \{\zeta_1,\zeta_2\}$.

 To show $\Omega_\lambda^\theta$ is a quasi-circle, since the
external rays and their preimages on $\Omega_\lambda^\theta$ are
analytic curves, it suffices to show that for any pair $z_1,z_2\in
\Omega_\lambda^\theta\setminus f_\lambda^{-1}(B_\lambda^{2L})$, the
turning $T(z_1,z_2):={\rm diam }(L(z_1,z_2))/|z_1-z_2|$ is bounded.
To this end, fix a small positive number $\varepsilon \ll\delta_0$ and
consider  $T(z_1,z_2)$ with $z_1,z_2\in
\Omega_\lambda^\theta\setminus f_\lambda^{-1}(B_\lambda^{2L})$ and
$|z_1-z_2|< \varepsilon$. There are two possibilities:

\textbf{Case 1. $L(z_1,z_2)\cap O_\lambda=\emptyset$}. In that case, by the structure of cut rays (see \cite{WQY}, Proposition 3.2 and Figure 3), we have
 $L(z_1,z_2)\subset \cup_{k\geq 0}f_\lambda^{-k}(B_\lambda)$ and
${\rm diam}(f_\lambda^k(L(z_1,z_2)))\rightarrow\infty$ as
$k\rightarrow\infty$. So there is an integer $\ell> 0$ such that
${\rm diam}(f_\lambda^\ell(L(z_1,z_2)))<\delta_0/2$ and ${\rm
diam}(f_{\lambda}^{\ell+1}(L(z_1,z_2)))\geq\delta_0/2$. By a suitable choice of
$L$, we may assume that either $f_\lambda^\ell(L(z_1,z_2))\subset f_\lambda^{-1}(B_\lambda^L)$ or $f_\lambda^\ell(L(z_1,z_2))\subset
\Omega_\lambda^{\tau^{\ell}(\theta)}\setminus
f_\lambda^{-1}(B_\lambda^{2L})$.

If $f_\lambda^\ell(L(z_1,z_2))\subset f_\lambda^{-1}(B_\lambda^L)$,
then the turning $T(f_\lambda^\ell(z_1), f_\lambda^\ell(z_2))$ is bounded by a constant $C_0$ since
$\Omega_\lambda^{\tau^{\ell}(\theta)}\cap B_\lambda$ is an analytic
curve. By Koebe distortion theorem,
$$T(z_1,z_2)\leq C_1T(f_\lambda^\ell(z_1), f_\lambda^\ell(z_2))\leq C_0C_1.$$

If $f_\lambda^\ell(L(z_1,z_2))\subset
\Omega_\lambda^{\tau^{\ell}(\theta)}\setminus
f_\lambda^{-1}(B_\lambda^{2L})$, then there exist two points
$u_1,u_2\in \overline{f_{\lambda}^\ell(L(z_1,z_2))}$ with ${\rm
diam}(f_{\lambda}^{\ell+1}(L(z_1,z_2)))=|f_{\lambda}(u_1)-f_{\lambda}(u_2)|\geq\delta_0/2$. Note that there is a constant $C_2>1$ such that for any $v_1,v_2\in \mathbb{\widehat{C}}\setminus f_\lambda^{-1}(B_\lambda^{2L})$, there is a smooth curve $\gamma(v_1,v_2)$ in $\mathbb{\widehat{C}}\setminus f_\lambda^{-1}(B_\lambda^{2L})$ connecting $v_1$ with $v_2$, with Euclidean length smaller than $C_2|v_1-v_2|$.
Thus
$$|f_{\lambda}(u_1)-f_{\lambda}(u_2)|=|\int_{\gamma(u_1,u_2)}f_{\lambda}'(z)dz|\leq\int_{\gamma(u_1,u_2)}|f_{\lambda}'(z)||dz|\leq C_2M|u_1-u_2|,$$
  where   $M=\max\{|f_{\lambda}'(z)|; z\in
\mathbb{\widehat{C}}\setminus f_\lambda^{-1}(B_\lambda^{2L})\}$.
It turns out that
$$\delta_0/(2C_2M)\leq {\rm
diam}(f_\lambda^\ell(L(z_1,z_2)))= {\rm
diam}(L(f_\lambda^\ell(z_1),f_\lambda^\ell(z_2)))<\delta_0/2.$$
 It follows (from the above fact) that there is a constant $c=c(\delta_0/(2C_2M))>0$ such that $|f_\lambda^\ell(z_1)-f_\lambda^\ell(z_2)|\geq c$.
 By Koebe distortion theorem,
$$T(z_1,z_2)\leq C_1 T(f_\lambda^\ell(z_1), f_\lambda^\ell(z_2))\leq\frac{C_1\delta_0}{2c}.$$

\textbf{Case 2. $L(z_1,z_2)\cap O_\lambda\neq\emptyset$}. In that case, there is
a smallest integer $\ell>0$  such that $0\in
f_\lambda^\ell(L(z_1,z_2))$. If ${\rm
diam}(f_\lambda^\ell(L(z_1,z_2)))<\delta_0/2<\min_{\zeta\in
\partial T_\lambda}|\zeta|/2$ (we may assume $\delta_0<\min_{\zeta\in
\partial T_\lambda}|\zeta|$), then
$f_\lambda^\ell(L(z_1,z_2))$ is contained in $\{|z|<\min_{\zeta\in
\partial T_\lambda}|\zeta|/2\}$ and the turning $T(f_\lambda^\ell(z_1), f_\lambda^\ell(z_2))$ is
bounded by a constant $C_0$. By Koebe distortion theorem,
$$T(z_1,z_2)\leq C_1T(f_\lambda^\ell(z_1), f_\lambda^\ell(z_2))\leq
C_0C_1.$$

If ${\rm diam}(f_\lambda^\ell(L(z_1,z_2)))\geq\delta_0/2$, then there
is an integer $m<\ell$ with ${\rm
diam}(f_\lambda^m(L(z_1,z_2)))<\delta_0/2$ and ${\rm
diam}(f_{\lambda}^{m+1}(L(z_1,z_2)))\geq\delta_0/2$.
With the same argument as that in Case 1, we conclude that $T(z_1,z_2)$ is bounded.
\end{proof}

Let $\Theta_{per}$ be a subset  of $\Theta\setminus \{1,1/2\}$, consisting of all periodic angles under the map $\tau$. One may verify that
 $\Theta_{per}$ is a dense subset of $\Theta$.

\begin{thm}[Cut rays with real parameters]\label{cutreal} For  any $\lambda\in(0,+\infty)$ and any  angle  $\theta\in\Theta_{per}$ with itinerary
 $(s_0,s_1,s_2,\cdots)$, the
set $$\Omega_\lambda^\theta:=\bigcap_{k\geq0} f_\lambda^{-k}((S_{s_k}^\lambda\cup S_{-s_k}^\lambda)\setminus\mathbb{R}^*)$$
is a Jordan curve intersecting the Julia set $J(f_\lambda)$ in
a Cantor set, where $\mathbb{R}^*=\mathbb{R}\setminus\{0\}$.
Moreover, if $\mathcal{F}_0\ni \lambda_j\rightarrow\lambda\in(0,+\infty)$,
 then $\Omega_{\lambda_j}^\theta\rightarrow \Omega_{\lambda}^\theta$ in Hausdorff topology.
\end{thm}

Here is a  remark. If $\lambda\in\mathcal{F}$, then $\cap_{k\geq0} f_\lambda^{-k}(S_{s_k}^\lambda\cup S_{-s_k}^\lambda)=\cap_{k\geq0} f_\lambda^{-k}((S_{s_k}^\lambda\cup S_{-s_k}^\lambda)\setminus\mathbb{R}^*)$, so the latter is also a reasonable definition of cut rays. However,
if $\lambda\in(0,+\infty)$, the set $\cap_{k\geq0} f_\lambda^{-k}(S_{s_k}^\lambda\cup S_{-s_k}^\lambda)$ is not  a Jordan curve in general.

The proof of Theorem \ref{cutreal} is essentially the same as that of Proposition 3.9 in \cite{WQY}. We would like to mention the idea of the proof here.
Let $Y_\lambda=\mathbb{\widehat{C}}\setminus ([-\infty, v_\lambda^+]\cup [v_\lambda^+,+\infty]\cup \overline{B_\lambda^L})$ for some large $L>1$ and $p$ be the period of $\theta$. The itinerary of $\theta$ satisfies
$s_{p+k}=s_k$ for all $k\geq0$. Since $\theta\neq1,1/2$, one of $s_k$ will be in the set $\{\pm1,\cdots,\pm(n-2)\}$ and for any $k\geq0$ and any
$(\epsilon_1,\cdots,\epsilon_p)=(\pm s_k, \cdots, \pm s_{k+p-1})$, the set $h^\lambda_{\epsilon_1}\circ\cdots\circ
h^\lambda_{\epsilon_{p}}(Y_\lambda)$ is compactly contained in $Y_\lambda$ (one should note that if $\theta=1$ or 1/2, then $h^\lambda_{1}\circ\cdots\circ h^\lambda_{1}(Y_\lambda)$ is not compactly contained in $Y_\lambda$). Similar to the proof of Proposition 3.9 in \cite{WQY}, one can construct two sequence of
Jordan curves converging  to the boundaries of the two components of $\mathbb{\widehat{C}}-\Omega_\lambda^\theta$. In this way $\Omega_\lambda^\theta$ is locally connected.  One can show that its two complement
components share the same boundary, so $\Omega_\lambda^\theta$ is a Jordan curve.

 With the same proof as  Lemma \ref{3e}, one can show that if $\theta\in\Theta_{per}$, then the cut ray $\Omega_\lambda^\theta$ is a holomorphic motion in a neighborhood of the real and positive axis. This yields  the continuity of cut rays. We omit the details.

\begin{pro}[Preimages of cut ray, \cite{WQY}, Prop 3.5]\label{23g} For any $\lambda\in\mathcal{F}_0$ and any
$\theta\in\Theta_{per}$, suppose that $(\Omega_\lambda^\theta-\{0,\infty\})\cap(\cup_{1\leq k\leq N}f_\lambda^k(C_\lambda))=\emptyset$
for some $N\geq1$. Then, for any $\alpha\in
\cup_{0\leq k\leq N}\tau^{-k}(\theta)$, there is a unique Jordan curve
$\Omega_\lambda^\alpha$ (or $\Omega_\lambda^{\alpha+1/2}$) containing $0$ and
$\infty$, such that  $f_\lambda(\Omega_\lambda^\alpha)=\Omega_\lambda^{\tau(\alpha)}$
and $R_\lambda(\alpha)\cup R_\lambda(\alpha+1/2)\subset \Omega_\lambda^\alpha\cap B_\lambda$.
\end{pro}

The Jordan curve $\Omega^\alpha_\lambda$ defined in Proposition \ref{23g} is also called a cut ray. We remark that the statement of Proposition \ref{23g}
is slightly different from  Prop 3.5 in \cite{WQY}, but their proofs are same.

\begin{rmk} \label{3gg} The cut ray $\Omega^\alpha_\lambda$ 
defined by Proposition \ref{23g} satisfies:


1. There is a neighborhood $\mathcal{U}$
 of $\lambda$, such that for all $u\in\mathcal{U}\cap \mathcal{F}_0$, $(\Omega_u^\theta-\{0,\infty\})\cap(\cup_{1\leq k\leq N}f_u^k(C_u))=\emptyset$ (this implies the cut ray $\Omega_u^\alpha$ exists). By Lemma \ref{3e} and Theorem \ref{cutreal}, the cut ray  $\Omega_u^\alpha$ moves continuously with respect to $u\in \mathcal{U}\cap \mathcal{F}_0$.

 2.  $\Omega_\lambda^\alpha$ is a quasi-circle (by Lemma \ref{3ef} or with the same proof).
\end{rmk}

\begin{lem}\label{3kk} For any $\lambda\in \mathcal{F}_0$ and any two different external rays $R_{\lambda}(t_1)$ and  $R_{\lambda}(t_2)$, there is a cut ray $\Omega_\lambda^\alpha$ with $\alpha\in \cup_{k\geq0}\tau^{-k}(\Theta_{per})$ separating them.
\end{lem}
\begin{proof} Since $\Theta_{per}$ is an infinite set, we can find an  angle $\theta\in \Theta_{per}$ such that
 $(\Omega_\lambda^\theta-\{0,\infty\})\cap(\cup_{k\geq1}f_\lambda^k(C_\lambda))=\emptyset$. The preimages $\cup_{k\geq0}\tau^{-k}(\theta)$ of $\theta$ are dense in the unit circle, so there is $\alpha\in\cup_{k\geq0}\tau^{-k}(\theta)$ lying in between $t_1$ and $t_2$. Then  $R_{\lambda}(t_1)$ and  $R_{\lambda}(t_2)$ are contained in different components of $\mathbb{\widehat{C}}-\Omega_\lambda^\alpha$.
\end{proof}

\section{$\partial\mathcal{H}_0$ is a Jordan curve}\label{0Jordan}

In this section, we will show that $\partial\mathcal{H}_0$ is a Jordan curve.  
We begin with a dynamical result for our purpose.
 To prove Theorem \ref{11a} in \cite{WQY},  we reduce the situation to the following:

\begin{thm}[Backward contraction on $\partial B_\lambda$,\cite{WQY}] \label{4a}  Suppose that $\lambda\in \mathbb{C}^*\setminus\mathcal{H}_0$ and $\partial B_\lambda$ contains neither
 a parabolic point nor the recurrent critical set $C_\lambda$, then $f_\lambda$ satisfies the following property on  $\partial B_\lambda$:  there exist three constants
$\delta_0>0$, $C>0$ and $0<\rho<1$ such that for any
$0<\delta<\delta_0$, any $z\in\partial B_\lambda$, any integer
$k\geq0$ and any component $U_k(z)$ of $f_\lambda^{-k}(B(z,\delta))$
that intersects with $\partial B_\lambda$, $U_k(z)$ is simply
connected with Euclidean diameter ${\rm diam}(U_k(z))\leq C \delta \rho^k$.
\end{thm}

 We refer the reader to
\cite{WQY} for a detailed proof based on  Yoccoz puzzle theory. (To obtain Theorem \ref{4a}, one should combine two results in  \cite{WQY}: Theorem 1.2 in Section 7.5
and Proposition 6.1 in Section 6.)

\begin{lem} \label{3a} Suppose that
$J(f_\lambda)$ is not a Cantor set. If $\partial B_\lambda$ contains
neither a critical point nor a parabolic cycle, then there exist an integer
$k\geq 1$ and two topological disks $U_\lambda,V_\lambda$ with
$\overline{B_\lambda}\subset V_\lambda\subset U_\lambda$, such that
$f^k_\lambda: V_\lambda\rightarrow U_\lambda$ is a polynomial-like
map of degree $n^k$ with only one critical point $\infty$.
\end{lem}
\begin{proof}
The map $f_\lambda$ satisfies the assumptions in Theorem \ref{4a}. This guarantees the existence of three constants $\delta_0, C, \rho$.

Let $N_\delta$ be $\delta$-neighborhood of $\partial B_\lambda$,
defined as the set of all points whose Euclidean distance to $\partial
B_\lambda$ is smaller than $\delta$. We choose  an integer $\ell>0$
and a number $\delta<\delta_0$ such that $C\rho^\ell<1$ and
$(\cup_{0\leq j<\ell}f_\lambda^{-j}(C_\lambda))\cap
N_{\delta}=\emptyset$.

Given an Jordan curve $\gamma$, we define its partial distance to
$\partial B_\lambda$ by $\varpi(\gamma):=\max_{z\in \gamma}d(z,
\partial B_\lambda)$, where $d(\cdot,\cdot)$ is Euclidean distance. We choose a Jordan curve $\gamma_0\subset \mathbb{\widehat{C}}\setminus\overline{B}_\lambda$ with
$\varpi(\gamma_0)<\delta$. The annulus between $\gamma_0$ and
$\partial B_\lambda$ is denoted by $A_0$. Since $(\cup_{0\leq
j<\ell}f_\lambda^{-j}(C_\lambda))\cap N_{\delta}=\emptyset$, there is an
annular component of $f_\lambda^{-\ell}(A_0)$, say $A_1$, with
$\partial B_\lambda$ as one of its  boundary components. The other boundary curve
is denoted by $\gamma_1$. Theorem \ref{4a} implies
$\varpi(\gamma_1)\leq \varpi(\gamma_0) C\rho^\ell<\delta$. Continuing
inductively, for any $k\geq1$, there is an annular component of
$f_\lambda^{-\ell}(A_{k-1})$, say $A_{k}$,  whose boundary curves are
$\partial B_\lambda$ and $\gamma_{k}$. Then we have
$$\varpi(\gamma_k)/\varpi(\gamma_0)\leq C\rho^{k\ell}.$$

So we can choose $k_0 >0$ such that $\varpi(\gamma_{k_0})<\min_{z\in
\gamma_0}d(z,
\partial B_\lambda)$. Let $V_\lambda$ be the unbounded component of
$\mathbb{\widehat{C}}-\gamma_{k_0}$ and $U_\lambda$ be the
unbounded component of $\mathbb{\widehat{C}}-\gamma_{0}$. Then
$f_{\lambda}^{k_0\ell}: V_\lambda\rightarrow U_\lambda$ is a
polynomial-like map of degree $n^{k_0\ell}$, with only one critical
point $\infty$. It is actually quasiconformally conjugate to the power
map $z\mapsto z^{n^{k_0\ell}}$.
\end{proof}

\begin{lem} \label{3b} Suppose $\lambda\in \partial
\mathcal{H}_0$, then  $\partial B_\lambda$   contains either  the critical set
$C_\lambda$ or a   parabolic cycle of $f_\lambda$.
\end{lem}

\begin{proof} If
 $\partial B_\lambda$ contains neither the  critical set $C_\lambda$ nor a parabolic cycle,
then it follows from Lemma \ref{3a} that there exist an integer $k\geq
1$ and two topological disks $U_\lambda,V_\lambda$ with
$\overline{B_\lambda}\subset V_\lambda\subset U_\lambda$, such that
$f^k_\lambda: V_\lambda\rightarrow U_\lambda$ is a polynomial like
map of degree $n^k$ with only one critical point $\infty$. We may
assume that $\overline{U_\lambda}$ has no intersection with  $\cup_{0\leq
j<k}f_\lambda^{-j}(C_\lambda)$.

Then there is a neighborhood of $\mathcal{U}$ of $\lambda$, such
that for all $u\in \mathcal{U}$, the set  $\cup_{0\leq
j<k}f_u^{-j}(C_u)$
has no intersection with $\overline{U_\lambda}$, thus the
component $V_u$ of  $f^{-k}_u(U_\lambda)$ that contains $\infty$ is
a disk. Since $\partial V_u$ moves holomorphically with respect to
$u\in \mathcal{U}$, we may shrink $\mathcal{U}$ to a little bit
 so that for all  $u\in \mathcal{U}$, $\partial V_u$ is contained
 in $U_\lambda$. Set $U_u=U_\lambda$. In this way, we get a polynomial-like map $f^k_u: V_u\rightarrow
 U_u$ with only one critical point $\infty$, for all $u\in
 \mathcal{U}$. As a consequence,  the Julia set $J(f_u)$ is not a Cantor set for $u\in
 \mathcal{U}$.

 But this is impossible since $\lambda\in \partial\mathcal{H}_0$.
\end{proof}

Given a parameter $\lambda\in \mathcal{F}$, if $C_\lambda\subset \partial B_\lambda$,
then there is a unique external ray $R_\lambda(t)$ landing at $v_\lambda^+$. We define $\theta(\lambda)=t$.
Note that $C_\lambda\subset \partial B_\lambda$ if and only if $v_\lambda^+\in \partial B_\lambda$.

\begin{lem} \label{angle} If  $\lambda\in \mathcal{F}$ and $v_\lambda^+\in \partial B_\lambda$, then $0<\theta(\lambda)<\frac{1}{2(n-1)}$.
\end{lem}

\begin{proof} If  $\lambda\in \mathcal{F}$, then  $v_\lambda^{+}$ is contained in the interior of $S_0^\lambda$. Note that $\Omega_\lambda^{1}\subset S_{n-1}^\lambda\cup S_{-(n-1)}^\lambda$
 and $\Omega_\lambda^{\frac{1}{2(n-1)}} \subset S_{1}^\lambda\cup S_{-1}^\lambda$, we have  $0<
\theta(\lambda)<\frac{1}{2(n-1)}$.
\end{proof}


\begin{lem}[\cite{WQY}, Prop 7.5] \label{3bb} If $\partial B_\lambda$   contains a parabolic cycle, then the following holds:

1. There is a symbol $\epsilon\in \{\pm1\}$, an integer $p\geq 1$, a critical point $c\in C_\lambda$
and two disks $U$ and $V$ containing $c$, such that  $\epsilon f_\lambda^p:
U\rightarrow V$ is a quadratic-like map, hybrid equivalent to the polynomial
$z\mapsto z^2+1/4$.

2. Let  $K$ be the  filled Julia set of $\epsilon f_\lambda^p:
U\rightarrow V$, then for any $j\geq 0$, then intersection
$f_\lambda^j(K)\cap \partial B_\lambda$ is a singleton.
\end{lem}

 Base on Lemma \ref{3bb},  let $K^+\in \{f_\lambda(K),
-f_\lambda(K)\}$ be the set containing $v_\lambda^+$,
and $\beta_\lambda$ be the intersection point of $K^+$ and
$\partial B_\lambda$.

\begin{rmk} \label{4pp}
If $n$ is odd, since $f_\lambda$ is an odd function,   $\beta_\lambda$ is necessarily a
parabolic point; if $n$ is even, either $\beta_\lambda$ or $-\beta_\lambda$ is
a parabolic point. Thus $\theta$ satisfies either $\tau^p(\theta)\equiv \theta$ or $\tau^p(\theta)\equiv \theta+\frac{1}{2}$ for some $p\geq1$.
\end{rmk}

For any $t\in [0, 1)$,
the parameter ray $\mathcal{R}_0(t)$ of angle $t$ in $\mathcal{H}_0$ is defined by $\mathcal{R}_0(t):=\Phi_0^{-1}((1,+\infty)e^{2\pi i t})$.
Its {\it impression } $\mathcal{X}_t$  is defined by
$$\mathcal{X}_t:=\cap_{k\geq1}\overline{\Phi_0^{-1}(\{re^{2\pi i \theta}; 1<r<1+{1}/{k},
|\theta-t|<{1}/{k}\}}).$$
 The  set $\mathcal{X}_t$  is  a connected and compact subset of $
\partial\mathcal{H}_0$. It satisfies
$$\mathcal{X}_{t+\frac{1}{n-1}}=e^{2\pi i/{(n-1)}}\mathcal{X}_t,  \ \{\overline{\lambda}; \lambda\in\mathcal{X}_{t}\}=\mathcal{X}_{1-t}.$$



\begin{lem} \label{3d} Let $t\in [0, \frac{1}{n-1})$ and  $\lambda\in
\mathcal{X}_t\cap \mathcal{F}_0$.

1. If $\lambda$ is not a cusp, then the external ray
$R_\lambda(t/2)$ lands at $v_\lambda^+$.

2. If $\lambda$ is a cusp, then the external ray $R_\lambda(t/2)$ lands at $\beta_\lambda$.
\end{lem}

\begin{proof} For any  parameter $\lambda\in
\mathcal{X}_t\cap \mathcal{F}_0$, 
it follows from Lemma \ref{3b} that either $C_\lambda\subset
\partial B_\lambda$ or $\partial B_\lambda$ contains a parabolic cycle. Since $\partial B_\lambda$ is a Jordan curve (Theorem \ref{11a}),
there is an external ray $R_\lambda(t')$ landing at $v_\lambda^+$ (if $\lambda$ is not a cusp) or
$\beta_\lambda$ (if $\lambda$ is a cusp).

If $t'\notin\{t/2, (1+t)/2\}$, then there exist two cut rays
$\Omega_\lambda^\alpha$ and $\Omega_\lambda^\beta$ with $\alpha,\beta\in \cup_{k\geq0}\tau^{-k}(\Theta_{per})$ (Lemma \ref{3kk}) such that the
connected set $R_\lambda(t')\cup \{v_\lambda^+\}$ (if $\lambda$ is not a cusp) or $R_\lambda(t')\cup K^+$ (if $\lambda$ is a cusp), and the external rays $
R_\lambda(t/2), R_\lambda((t+1)/2)$ are contained in three different
components of
$\mathbb{\widehat{C}}\setminus(\Omega_\lambda^\alpha\cup\Omega_\lambda^\beta)$.
See Figure 4.
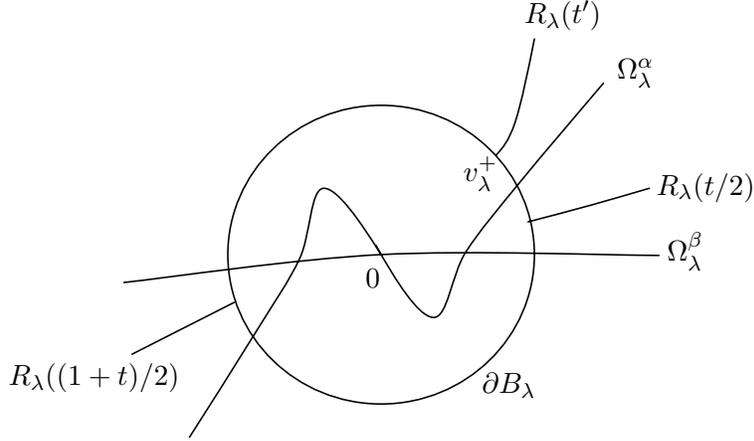
\begin{figure}[h]
\begin{center}
\includegraphics[height=6cm]{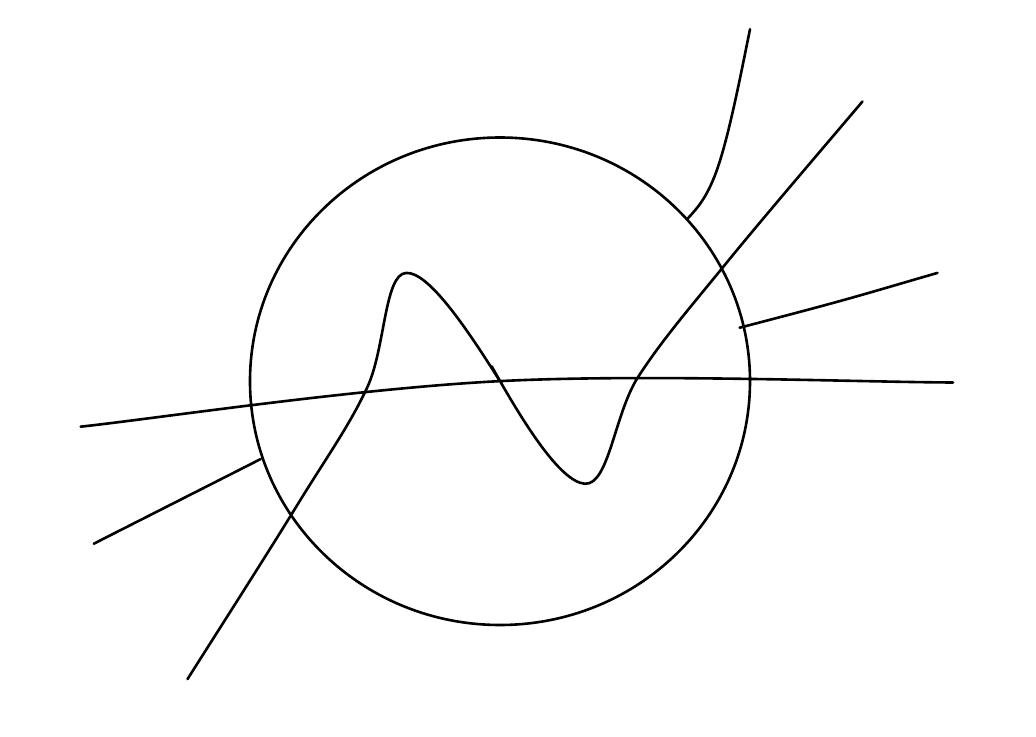}
\put(-130,70){$0$} \put(-35,149){$\Omega_\lambda^\alpha$}
\put(-16,80){$\Omega_\lambda^\beta$} \put(-20,105){$R_\lambda(t/2)$}
\put(-70,170){$R_\lambda(t')$} \put(-265,33){$R_\lambda((1+t)/2)$}
\put(-86,30){$\partial B_\lambda$} \put(-93,110){$v_\lambda^+$}
 \caption{Two cut rays
$\Omega_\lambda^\alpha$ and $\Omega_\lambda^\beta$ separate the
external rays $R_\lambda(t'), R_\lambda(t/2),  R_\lambda((t+1)/2)$
in case that $\lambda$ is not a cusp.}
\end{center}\label{f5}
\end{figure}
 Since the critical value $v_u^+=2\sqrt{u}$ and the cut rays $\Omega_u^\alpha, \Omega_u^\beta$  move continuously with respect to the
parameter $u\in\mathcal{F}_0$ (Lemma \ref{3e} and Remark \ref{3gg}),  there is a neighborhood
$\mathcal{V}$ of $\lambda$  such that for all $u\in
\mathcal{V}\cap\mathcal{F}_0$,

$\bullet$  $R_u(t')$ and  $v_u^+$ are contained in the same
component of
$\mathbb{\widehat{C}}\setminus(\Omega_u^\alpha\cup\Omega_u^\beta)$.

$\bullet$ The external rays $R_u(t'), R_u(t/2), R_u((t+1)/2)$ are
contained in three different components of
$\mathbb{\widehat{C}}\setminus(\Omega_u^\alpha\cup\Omega_u^\beta)$.

By shrinking $\mathcal{V}$ a little bit, we see that there is a small number $\varepsilon>0$ such that $\arg\Phi_0(u)=2 \arg
\phi_u(v_u^+)\notin (t-\varepsilon,t+\varepsilon)$ for all  $u\in
\mathcal{V}\cap\mathcal{F}_0\cap \mathcal{H}_0$.
It's a contradiction since $\lambda\in
\mathcal{X}_t$.

So either $t'=t/2$ or  $t'=(1+t)/2$. To finish, we show the latter
is impossible.
If $\lambda\in (0,+\infty)$, then $\partial B_\lambda$ contains a cusp and $t'=0$. If $\lambda\in \mathcal{F}$,
then there is a component $V$ of $\mathbb{\widehat{C}}\setminus (\Omega_\lambda^1\cup \Omega_\lambda^{\frac{1}{2(n-1)}})$ such that
 $v_\lambda^+\cup R_\lambda(t')\subset \overline{V}$. In this case,
we have $0\leq t\leq\frac{1}{2(n-1)}$.

 So $t'=t/2$.
\end{proof}







We are now ready to state the main result of this section:

\begin{pro} [Main proposition]\label{3ff} Given two  parameters $\lambda_1, \lambda_2\in \mathcal{F}$, if
$v_{\lambda_i}^+\in \partial B_{\lambda_i}$(i=1,2) and
  $\theta(\lambda_1)=\theta(\lambda_2)$, then $\lambda_1=\lambda_2$.
\end{pro}

To prove  Proposition \ref{3ff}, we need the following result:

\begin{thm}[Lebesgue measure] \label{lm} If $f_\lambda^k(v_\lambda^+)\in \partial B_\lambda$ for some $k\geq0$, then the Lebesgue measure of $J(f_\lambda)$
is zero.
\end{thm}

The proof of Theorem \ref{lm} is based on the Yoccoz puzzle theory following Lyubich \cite{L}. For this, we put the proof in the appendix.

\

\noindent{\it Proof of Proposition \ref{3ff}.}  If  $\theta(\lambda_1)$ is a rational number, then both  $f_{\lambda_1}$ and $f_{\lambda_2}$ are postcritically finite.
We define a homeomorphism
$\psi:\mathbb{\widehat{C}}\rightarrow \mathbb{\widehat{C}}$ such that
$\psi|_{B_{\lambda_1}}=\phi_{\lambda_2}^{-1}\circ \phi_{\lambda_1}$. Then
there is a homeomorphism $\varphi:\mathbb{\widehat{C}}\rightarrow \mathbb{\widehat{C}}$ satisfying $\psi\circ f_{\lambda_1}=
f_{\lambda_2}\circ \varphi$ and
$\varphi|_{{B_{\lambda_1}}}=\psi|_{{B_{\lambda_1}}}$.  (In fact, $\varphi$ and $\psi$ can be made quasiconformal because  $\partial B_1$ and $\partial B_2$ are quasi-circles, see Theorem \ref{11a}.) The condition $\theta(\lambda_1)=\theta(\lambda_2)$ implies that $\varphi$ and $\psi$ are isotopic rel the postcritical
set $P(f_{\lambda_1}):=\{\infty\}\cup\cup_{k\geq1}f_{\lambda_1}^k(C_{\lambda_1})$. Thus $f_{\lambda_1}$ and  $f_{\lambda_2}$ are
combinatorially equivalent. It follows from Thurston's theorem (see \cite{DH}) that
$f_{\lambda_1}$ and $f_{\lambda_2}$ are  conjugate via a M\"obius transformation.
This M\"obius map takes the form  $\gamma(z)=az$ with $a^{n-1}=1$ and $\lambda_2=a^2\lambda_1$.
 The condition
$\lambda_1, \lambda_2\in \mathcal{F}$ implies
$\lambda_1=\lambda_2$.

In the following, we assume $\theta(\lambda_1)$ is an irrational number. In that case, both  $f_{\lambda_1}$ and $f_{\lambda_2}$
are postcritically infinite. We will construct a quasiconformal conjugacy between $f_{\lambda_1}$ and $f_{\lambda_2}$
with the help of cut rays.


By Lemma \ref{3ef},  periodic cut rays are quasi-circles. This enables us to construct  a
quasi-conformal map $\psi_0:
\mathbb{\widehat{C}}\rightarrow\mathbb{\widehat{C}}$ with
$0,\infty$ fixed, such that:

$\bullet$ $\psi_0|_{B_{\lambda_1}^L}=\phi_{\lambda_2}^{-1}\circ
\phi_{\lambda_1}|_{B_{\lambda_1}^L}$.

$\bullet$
$\psi_0(\Omega_{\lambda_1}^{1})=\Omega_{\lambda_2}^{1}$.

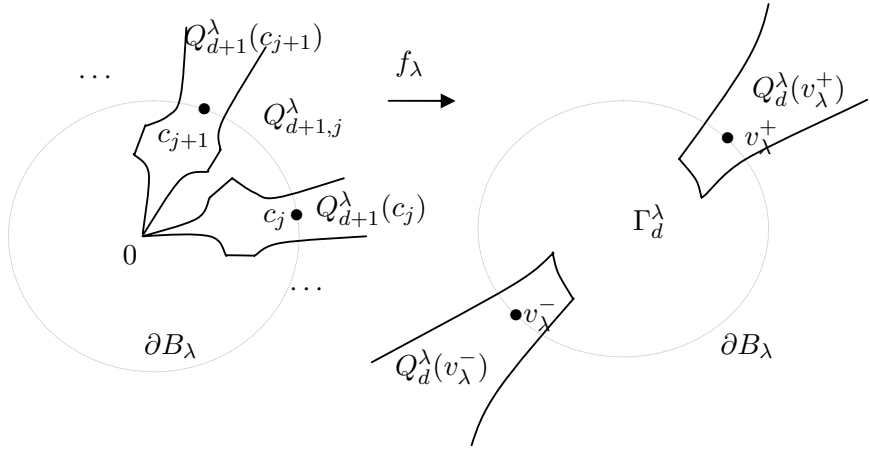
\begin{figure}[h]
\begin{center}
\includegraphics[height=7cm]{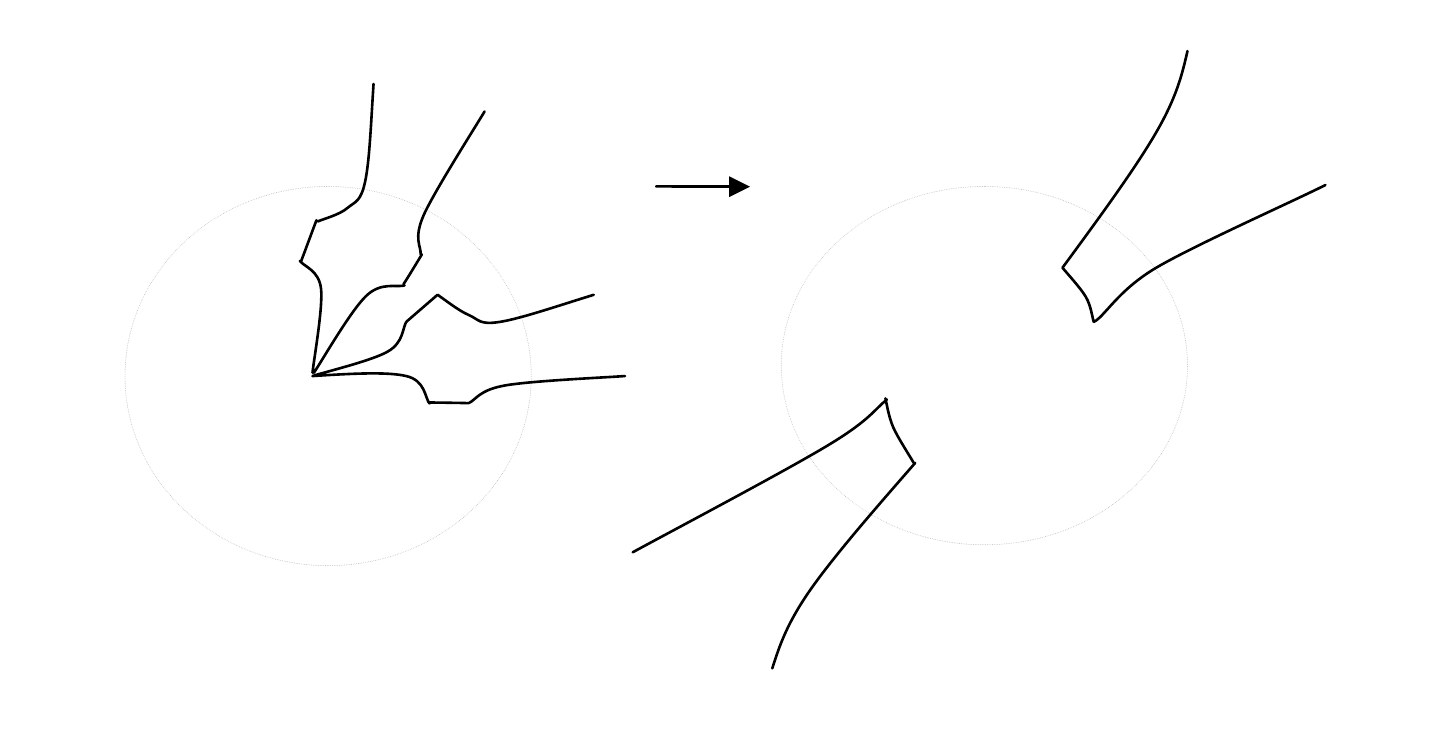}
\put(-210,160){$f_\lambda$} \put(-120,100){$\Gamma_d^\lambda$}
\put(-167,65){$\bullet  v_\lambda^-$} \put(-210,45){$
Q_d^\lambda(v_\lambda^-)$}
 \put(-87,132){$\bullet \ v_\lambda^+$}  \put(-75,150){$Q_d^\lambda(v_\lambda^+)$}
 \put(-240,105){$Q_{d+1}^\lambda(c_j)$}
 \put(-260,140){$Q_{d+1,j}^\lambda$}   \put(-290,170){$Q_{d+1}^\lambda(c_{j+1})$}
 \put(-250,103){$\bullet$} \put(-260,103){$c_j$}
  \put(-313,87){$0$}
  \put(-305,53){$\partial B_\lambda$}  \put(-87,53){$\partial B_\lambda$}
   \put(-285,143){$\bullet$}\put(-300,133){$c_{j+1}$}
  \put(-250,75){$\cdots$}
 \put(-330,155){$\cdots$}
\caption{Partition and labeling.}
\end{center}\label{f5}
\end{figure}

In the following, we will construct a sequence of quasi-conformal maps $\psi_j$
such that

(a). $f_{\lambda_2}\circ \psi_{j+1}=\psi_j\circ f_{\lambda_1}$
for all $j\geq0$,

(b). $\psi_{j+1}|_{f_{\lambda_1}^{-j}(B_{\lambda_1}^L)}=\psi_{j}|_{f_{\lambda_1}^{-j}(B_{\lambda_1}^L)}$,

(c). $\psi_j(\Omega_{\lambda_1}^{\alpha})=\Omega_{\lambda_2}^{\alpha}$ for all
 $\alpha\in\tau^{-j}\{1,\frac{1}{2}\}$.

 The construction is as follows.
For $\lambda\in\{\lambda_1,\lambda_2\}$, any $d\geq0$ and any $z\in
C_\lambda\cup \{v_\lambda^+,v_\lambda^-\}$, let $Q_d^\lambda(z)$ be
the component of $\overline{\mathbb{C}}\setminus
f_\lambda^{-d}(\Omega_\lambda^1)$ containing $z$.
 The domain
$\Gamma_d^\lambda:=\overline{\mathbb{C}}
\setminus(\overline{Q^\lambda_d(v_\lambda^+)\cup
Q^\lambda_d(v_\lambda^-)})$ either is empty or  consists of one or
two topological disks. Each component of
$f_\lambda^{-1}(\overline{\mathbb{C}} \setminus\Gamma_d^\lambda)$ is
a disk. Let $Q^\lambda_{d+1,j}$ be its component lying in between
$Q^\lambda_{d+1}(c_j(\lambda))$ and $Q^\lambda_{d+1}(c_{j+1}(\lambda))$ for $0\leq j
<2n$, $c_{2n}(\lambda)=c_0(\lambda)$. See Figure 5. Note that the map
$f_\lambda|_{Q^\lambda_{d+1,j}}:Q^\lambda_{d+1,j}\rightarrow\Gamma_d^\lambda$
is a conformal isomorphism.

Suppose that $\psi_0,\psi_1,\cdots, \psi_d$ are already defined and satisfy (a),(b), (c). We will
define $\psi_{d+1}$ piece by piece. Set  $\psi_{d+1}|_{Q^{\lambda_1}_{d+1,j}}=
(f_{\lambda_2}|_{Q^{\lambda_2}_{d+1,j}})^{-1}\circ \psi_d\circ
(f_{\lambda_1}|_{Q^{\lambda_1}_{d+1,j}})$. We then define
$\psi_{d+1}|_{\overline{Q^{\lambda_1}_{d+1}(c_j)}}$ so that it coincides with
$\psi_{d+1}|_{Q^{\lambda_1}_{d+1,j}}$ in their common
boundary and the following diagram commutes
$$
\xymatrix{&
 \overline{Q^{\lambda_1}_{d+1}(c_j(\lambda_1))}\ar[r]^{f_{\lambda_1}}\ar[d]_{\psi_{d+1}}
 &  \overline{Q^{\lambda_1}_{d}(f_{\lambda_1}(c_j(\lambda_1)))} \ar[d]_{\psi_d}\\
&\overline{Q^{\lambda_2}_{d+1}(c_j(\lambda_2))}\ar[r]_{f_{\lambda_2}}
&\overline{Q^{\lambda_2}_{d}(f_{\lambda_2}(c_j(\lambda_2)))}
 }
$$

One may verify that $\psi_{d+1}$ is well defined and satisfies
$f_{\lambda_2}\circ \psi_{d+1}=\psi_d\circ f_{\lambda_1}$.
By induction assumption, $\psi_d$ preserves the $d$-th preimages of $\Omega_\lambda^1$.  Then the condition $\theta(\lambda_1)=\theta(\lambda_2)$ and the construction of
$\psi_{d+1}$ implies that $\psi_{d+1}$  preserves the $(d+1)$-th preimages of $\Omega_\lambda^1$. Equivalently, for any
 $\alpha\in\tau^{-d-1}\{1,\frac{1}{2}\}$,  we have  $\psi_{d+1}(\Omega_{\lambda_1}^{\alpha})=\Omega_{\lambda_2}^{\alpha}$.
The equality
 $\psi_{d+1}|_{f_{\lambda_1}^{-d}(B_{\lambda_1}^L)}=\psi_{d}|_{f_{\lambda_1}^{-d}(B_{\lambda_1}^L)}$
 follows by induction.

The maps $\psi_j$ form a normal family since their dilations are
uniformly bounded above.  Let  $\psi_\infty$ be the limit map of  $\psi_j$.
It is holomorphic in the Fatou set $F(f_{\lambda_1})=\cup_k f_{\lambda_1}^{-k}(B_{\lambda_1}^L)$ and  satisfies
$f_{\lambda_2}\circ \psi_{\infty}=\psi_\infty\circ f_{\lambda_1}$ in $F(f_{\lambda_1})$.
By continuity,    $\psi_\infty \circ f_{\lambda_1}=f_{\lambda_2}\circ\psi_\infty$ in $\mathbb{\widehat{C}}$.
By Theorem \ref{lm}, the
Lebesgue measure of $J(f_{\lambda_1})$ is zero, so $\psi_\infty$ is a
M\"obius map of the form $\psi_\infty(z)=az$. One may verify that
$a^{n-1}=1$ and $\lambda_2=a^2\lambda_1$.
The condition $\lambda_1,\lambda_2\in
\mathcal{F}$ implies
$\lambda_1=\lambda_2$. \hfill $\Box$


\begin{thm} \label{3g} $\partial \mathcal{H}_0$ is a Jordan curve.
\end{thm}
\begin{proof}
We first show that $\mathcal{X}_0$ is a singleton. 
To do this, first note that the parameter ray $\mathcal{R}_0(0)$ is contained in the real and positive axis. So $\mathcal{X}_0$ contains at least one  positive number. We define $g_\lambda(z)=z^n(f_\lambda(z)-z)=z^{2n}-z^{n+1}+\lambda$ for $\lambda,z>0$.
The positive critical point of $g_\lambda$ is $z_*=(\frac{n+1}{2n})^{\frac{1}{n-1}}$ and for all $z>z_*$, we have $g_\lambda'(z)>0$.
Let $\lambda_*$ solve $g_{\lambda_*}(z_*)=0$, then $\lambda_*=\frac{n-1}{2n}(\frac{n+1}{2n})^{\frac{n+1}{n-1}}$. For any $\lambda>\lambda_*$,
we have $g_\lambda>0$. In this case, for any $z>0$, we have $f_\lambda^k(z)\rightarrow\infty$ as $k\rightarrow\infty$. This implies $[0,+\infty)\subset B_\lambda$. In particular, $v_\lambda^+\in B_\lambda$. Thus $(\lambda_*,+\infty)\subset\mathcal{R}_0(0)$. On the other hand, we have
$f_{\lambda_*}(z_*)=z_*$ and $f_{\lambda_*}'(z_*)=1$. So $\lambda_*$ is a cusp and $\lambda_*\in \mathcal{X}_0$.
Moreover, by elementary properties of real functions, there is a small number $\epsilon>0$ such that for all $\lambda\in(\lambda_*-\epsilon,\lambda_*)$,
 the map $f_\lambda$ has an attracting cycle. So $(\lambda_*-\epsilon,\lambda_*)$ is contained in a hyperbolic component (see Theorem \ref{6a}) and $(\lambda_*-\epsilon,\lambda_*)\cap\mathcal{X}_0=\emptyset$.
   If $\mathcal{X}_0\setminus \{\lambda_*\}\neq\emptyset$, then there is $\lambda\in \mathcal{X}_0\cap\mathcal{F}$
 which is not a cusp. By Lemma \ref{angle}, we have $0<\theta(\lambda)<\frac{1}{2(n-1)}$. However by Lemma \ref{3d}, we have
 $\theta(\lambda)=0$. This leads to a contradiction.

In the following, we assume $t\in(0,\frac{1}{n-1})$. Take two parameters  $\lambda_1,\lambda_2\in\mathcal{X}_t\cap\mathcal{F}$ which are not cusps,  it follows from  Lemma \ref{3d} that $\theta(\lambda_1)=\theta(\lambda_2)=t/2$. By Proposition \ref{3ff} we have $\lambda_1=\lambda_2$. Since there are countably many cusps,  the impression $\mathcal{X}_t$ is necessarily a
singleton. So $\partial \mathcal{H}_0$
is locally connected.

If there are two different angles $t_1,t_2\in
[0,\frac{1}{n-1})$ with
$\mathcal{X}_{t_1}=\mathcal{X}_{t_2}=\{\lambda\}$, then by Lemma
\ref{3d}, the external rays $R_\lambda(t_1/2)$ and
$R_\lambda(t_2/2)$ land at the same point on $\partial B_\lambda$.
But this is a contradiction since $\partial B_\lambda$ is a Jordan
curve (Theorem \ref{11a}).
\end{proof}

Theorem \ref{3g} has several consequences.
First, one gets a canonical parameterization $\nu:\mathbb{S}\rightarrow \partial \mathcal{H}_0$, where
$\nu(\theta)$ is defined to be the landing point of the parameter ray $\mathcal{R}_0(\theta)$ (namely,
$\nu(\theta):=\lim_{r\rightarrow 1^+}\Phi_0^{-1}(re^{2\pi i \theta}$)).

\begin{thm}\label{cuspc} $\nu(\theta)$ is a cusp if and only if $\theta$ is $\tau$-periodic.
\end{thm}
\begin{proof}

By Theorem \ref{3g}, we see that
$\nu(0)=\frac{n-1}{2n}(\frac{n+1}{2n})^{\frac{n+1}{n-1}}$ is a cusp.

Note that  $\nu(\theta+\frac{1}{n-1})=e^{2\pi i/(n-1)}\nu(\theta)$ and $(-1)^nf_{\nu(\theta+\frac{1}{n-1})}({e}^{\pi i/(n-1)}z)
={e}^{\pi i/(n-1)}f_{\nu(\theta)}
(z)$, thus $\nu(\theta)$ is a cusp if and only if $\nu(\theta+\frac{1}{n-1})$ is a cusp.
For this, we assume  $\theta\in (0,\frac{1}{n-1})$.

 If $\nu(\theta)$  is a cusp, then by Lemma \ref{3d}, the external ray
$R_{\nu(\theta)}(\theta/2)$ lands at $\beta_{\nu(\theta)}$. By Remark \ref{4pp},  $\frac{\theta}{2}$ satisfies either $\tau^p(\frac{\theta}{2})\equiv \frac{\theta}{2}$ or $\tau^p(\frac{\theta}{2})\equiv \frac{\theta}{2}+\frac{1}{2}$ for some $p\geq1$.
In either case, $\theta$ is $\tau$-periodic.

Conversely, we assume $\theta$ is $\tau$-periodic.
 If $\nu(\theta)$ is not a cusp, then by Lemma \ref{3d}, the external ray
$R_{\nu(\theta)}(\frac{\theta}{2})$ lands at $v_{\nu(\theta)}^+$. Note that $\frac{\theta}{2}$
satisfies either $\tau^p(\frac{\theta}{2})=\frac{\theta}{2}$ or
$\tau^p(\frac{\theta}{2})=\frac{\theta}{2}+\frac{1}{2}$ for some $p\geq1$. We have that either $f_{\nu(\theta)}^p(v_{\nu(\theta)}^+)=v_{\nu(\theta)}^+$ or
$f_{\nu(\theta)}^p(v_{\nu(\theta)}^+)=v_{\nu(\theta)}^-$. In the former case, we get a periodic critical point $c\in f_{\nu(\theta)}^{-1}(v_{\nu(\theta)}^+)$; in the latter case, we get a periodic critical point $c\in f_{\nu(\theta)}^{-1}(v_{\nu(\theta)}^-)$. These critical points will be in the Fatou set.
But this  contradicts  $v_{\nu(\theta)}^+\in \partial B_{\nu(\theta)}$.
\end{proof}

\begin{rmk} \label{classi} As a consequence of Lemma \ref{3d} and Theorem \ref{cuspc},

1. If $\theta$ is  $\tau$-periodic, then $\nu(\theta)$ is a cusp;

2. If $\theta$ is rational but not $\tau$-periodic, then $f_{\nu(\theta)}$ is postcritically finite;

3. If $\theta$ is irrational, then $f_{\nu(\theta)}$ is postcritically infinite.

In the last two cases, one has $C_{\nu(\theta)}\subset \partial B_{\nu(\theta)}$. Moreover, by Borel's normal number theorem, for almost all
$\theta\in(0,1]$, we have $\overline{\cup_{k\geq1}f_{\nu(\theta)}^k(C_{\nu(\theta)})}=\partial B_{\nu(\theta)}$.
\end{rmk}

\begin{pro}\label{hm} Set $\partial B_0=\mathbb{S}$ and $\mathcal{V}=\mathbb{C}\setminus \overline{\mathcal{H}_0}$, then there is a holomorphic motion $H:\mathcal{V}\times \mathbb{S}\rightarrow\mathbb{C}$
parameterized by $\mathcal{V}$ and with base point $0$ such that
$H(\lambda, \mathbb{S})=\partial B_\lambda$ for all $\lambda\in \mathcal{V}$.
\end{pro}

\begin{proof}
We first prove that every repelling periodic point  of $f_0(z)=z^n$
moves holomorphically in $\mathcal{H}_2\cup\{0\}$.  Let $z_0\in\mathbb{S}=J(f_0)$ be
such a point with period $k$. For small $\lambda$, the map $f_\lambda$ is a
 perturbation of $f_0$. By implicit function theorem, there is
a neighborhood $\mathcal{U}_0$ of $0$ such that $z_0$ becomes a repelling
point $z_\lambda$ of $f_\lambda$ with the same period $k$, for all $\lambda\in \mathcal{U}_0$.
On the other hand, for all $\lambda\in\mathcal{H}_2$,  each repelling cycle of $f_\lambda$ moves
holomorphically throughout $\mathcal{H}_2$ (see \cite{Mc4}, Theorem 4.2).

Since $\mathcal{H}_2\cup\{0\}$ is simply connected, by Monodromy theorem, there is a holomorphic map
$Z_{z_0}:\mathcal{H}_2\cup\{0\}\rightarrow \mathbb{C}$ such that $Z_{z_0}(\lambda)=z_\lambda$ for $\lambda\in
\mathcal{U}_0$.  Let $\textup{Per}{(f_0)}$ be all  repelling periodic points of $f_0$.
One may verify that 
  the map $y:\mathcal{H}_2\cup\{0\}\times \textup{Per}{(f_0)}\rightarrow
\mathbb{C}$ defined by $y(\lambda,z)=Z_{z}(\lambda)$ is a holomorphic motion.
Note that $\mathbb{S}=\overline{\textup{Per}{(f_0)}}$, by $\lambda$-Lemma  (see \cite{MSS} or \cite{Mc4}), there is an extension of $y$, say
$Y:\mathcal{H}_2\cup\{0\}\times \mathbb{S}\rightarrow\mathbb{C}$. It's obvious
that $Y(\lambda, \mathbb{S})$ is a connected component of $J(f_\lambda)$.

Now, we show $Y(\lambda, \mathbb{S})=\partial B_\lambda$ for all $\lambda\in
\mathcal{H}_2\cup\{0\}$. By the uniqueness of the holomorphic motion of
hyperbolic Julia sets, it suffices to show $Y(\lambda,
\mathbb{S})=\partial B_\lambda$ for small and real parameter $\lambda\in
(0,\epsilon)$, where $\epsilon>0$. To see this, note that when $\lambda\in(0,\epsilon)$
 the fixed point $p_0=1$ of $f_0$ becomes the repelling fixed points
$p_\lambda$ of $f_\lambda$, which is real and close to $1$. The map $f_\lambda$ has
exactly two real and positive fixed points. One is $p_\lambda$ and the
other is $p^*_\lambda$, which is near $0$.  It's obvious that $p_{\lambda}$ is
the landing point of the zero external ray of $f_\lambda$. So
$Y(\lambda,1)=p_{\lambda}\in\partial B_\lambda$. This implies $Y(\lambda,
\mathbb{S})=\partial B_\lambda$ for all $\lambda\in (0,\epsilon)$.

By above argument and Carath\'eodory convergence theorem, the map  $h:\mathcal{V}\times (\mathbb{\widehat{C}}\setminus
\mathbb{\overline{D}})$ defined by $h(u,z)=\phi_u^{-1}(z)$ if $u\in
\mathcal{U}\setminus\{0\}$ and $h(0,z)=z$.
 is a holomorphic motion of $\mathbb{\widehat{C}}\setminus
\mathbb{\overline{D}}$ when $u$
varies in $\mathcal{V}$.
By Slodkowski's theorem (see \cite{GJW} or \cite{Sl}), there is a holomorphic motion $H:\mathcal{V}\times
\mathbb{\widehat{C}}$ extending $h$ and for any $v\in \mathcal{V}$,
we have $H(v,\mathbb{S})=\partial B_v$.
\end{proof}

\begin{thm}\label{boundary} $\lambda\in \partial \mathcal{H}_0$ if and
only if $\partial B_\lambda$ contains either the critical set
$C_\lambda$ or a parabolic cycle of  $f_\lambda$.
\end{thm}
\begin{proof} By Lemma \ref{3b}, it suffices to prove the `if' part.

We first assume that $f_\lambda$ has a parabolic cycle on $\partial
B_\lambda$. By Lemma \ref{3bb},  the Julia set
$J(f_\lambda)$ contains a quasiconformal copy of the Julia set  of $z\mapsto z^2+1/4$. So the boundary $\partial
B_\lambda$ is not a quasi-circle. It follows from Proposition \ref{hm} that for all $v\in \mathbb{C}\setminus \overline{\mathcal{H}_0}$,
$\partial B_v$ is a quasi-circle. Thus  $\lambda\in\partial\mathcal{H}_0$.

Now assume $\lambda\in \mathcal{F}$ and  $C_\lambda\subset
\partial B_\lambda$. Recall that $\theta(\lambda)$ is defined such that
$R_\lambda(\theta(\lambda))$ lands at $v_\lambda^+$. By Lemma \ref{angle}, we have
$0< \theta(\lambda)<\frac{1}{2(n-1)}$.

  Similar to the proof of Theorem \ref{cuspc}, we conclude that $2\theta(\lambda)$ is not $\tau$-periodic.
Then  $\lambda'=\nu(2\theta(\lambda))\in\mathcal{F}$ is not a cusp (Theorem \ref{cuspc}). It satisfies $\theta(\lambda')=\theta(\lambda)$.
It follows from Proposition \ref{3ff} that $\lambda=\lambda'\in
\partial \mathcal{H}_0$.
\end{proof}


\section{Sierpi\'nski holes are Jordan domains}\label{sier}

Besides $\mathcal{H}_0$, there are two kinds of escape domains: the McMullen domain $\mathcal{H}_2$ and the Sierpinski locus $\mathcal{H}_k, k\geq3$. In \cite{D1}, Devaney  showed that the boundary $\partial \mathcal{H}_2$ is a Jordan curve by constructing of a sequence of analytic curves converging to it. In this section, we will show that the boundary of  every Sierpinski hole is  a Jordan curve. We remark that our approach also applies to $\partial\mathcal{H}_2$.  This will yield a different proof from Devaney's. An interesting fact is that our  proof relies on the boundary regularity of  $\partial \mathcal{H}_0$.

Let ${\mathcal{H}}$ be a escape domain of level $k\geq3$. It has no intersection with $\mathbb{R}^+:=(0,+\infty)$. (In fact, by elementary properties of real functions, one may verify that there is a positive parameter $\lambda^*\in(0,\nu(0))$ such that $(0,\lambda^*)\subset \mathcal{H}_2$ and for all $\lambda\in[\lambda^*, \nu(0)]$, the critical orbit of $f_{\lambda}$ remains bounded,  in that case, $f_\lambda$ is renormalizable,  see \cite{WQY} Lemma 7.5).

The relation
 ${e}^{\pi i/(n-1)}f_\lambda
(z)=(-1)^nf_{{e}^{2\pi i/(n-1)}\lambda}({e}^{\pi i/(n-1)}z)$
 implies that $e^{2\pi/(n-1)}\mathcal{H}_k=\mathcal{H}_k$.
 So we may assume $\mathcal{H}\subset \mathcal{F}$. The relation  $\overline{f_\lambda(\bar{z})}=f_{\bar{\lambda}}(z)$ implies that $\mathcal{H}_k$ is symmetric about the real axis. We may assume further:
 either $\mathcal{H}$ is symmetric about $\{\lambda\in \mathbb{C}^*; \arg\lambda =\frac{\pi}{n-1}\}$ or $\mathcal{H}\subset \{\lambda\in \mathbb{C}^*; 0<\arg \lambda <\frac{\pi}{n-1} \}$.

   The parameter ray $\mathcal{R}_{\mathcal{H}}(t)$ of angle $t\in(0,1]$ in $\mathcal{H}$ is defined by $\mathcal{R}_\mathcal{H}(t):=\Phi_\mathcal{H}^{-1}((0,1)e^{2\pi i t})$, its impression  $\mathcal{X}_{\mathcal{H}}(t)$   is defined by
$$\mathcal{X}_{\mathcal{H}}(t):=\cap_{j\geq1}\overline{\Phi_\mathcal{H}^{-1}(\{re^{2\pi i \theta}; 1-1/j<r<1,
|\theta-t|<{1}/{j}\}}).$$

When $\lambda$ ranges over $\mathcal{\overline{H}}$, the preimages $f_\lambda^{2-k}(0)$ move continuously and $f^{k-2}_\lambda$ maps each component of  $f^{2-k}_\lambda(T_\lambda)$ conformally onto $T_\lambda$.
 Let $U_\lambda$ be the component of $f^{2-k}_\lambda(T_\lambda)$ containing $v_\lambda^+$ and $g_\lambda$ be the inverse of $f^{k-2}_\lambda|_{U_\lambda}$. Both $g_\lambda(0)$ and
$U_\lambda$ moves continuously for $\lambda\in \mathcal{\overline{H}}$ (and holomorphically in $ \mathcal{H}$). The internal ray $R_{U_\lambda}(t)$ of angle $t$ in $U_\lambda$ is defined by $R_{U_\lambda}(t):=g_\lambda(R_{T_\lambda}(t))$.

\begin{lem}\label{5a} For  any integer $p\geq0$, the set $f_\lambda^{-p}(\overline{B}_\lambda)$ moves continuously (in Hausdorff topology) with respect to $\lambda\in \mathbb{C}^*\setminus \overline{\mathcal{H}_0}$.
\end{lem}
\begin{proof} It  is an  immediate consequence of Proposition \ref{hm}.
\end{proof}






\begin{lem}\label{5b} For any $t\in[0,1)$ and any $\lambda\in\mathcal{X}_{\mathcal{H}}(t)\setminus\partial\mathcal{H}_0$,  we have $v_\lambda^+\in \partial U_\lambda$
and the internal ray $R_{U_\lambda}(t)$ lands at $v_\lambda^+$.
\end{lem}
\begin{proof}
  It follows from Lemma \ref{5a} 
   that the closure of the external ray  ${R_{\lambda}(t)}$ moves continuously (in Hausdorff topology) for
   $\lambda\in \overline{\mathcal{H}}\setminus\partial\mathcal{H}_0$. Note that pulling back $\overline{R_{\lambda}(t)}$ via $f_\lambda^p$ preserves the continuity.
\end{proof}

\begin{pro}\label{5c} For any $t\in[0,1)$, the set $\mathcal{X}_{\mathcal{H}}(t)\setminus\partial\mathcal{H}_0$ is either empty or a singleton.
\end{pro}

\begin{proof}  If not, there exist $t\in[0,1)$ and a connected and compact subset $\mathcal{E}$ of $\mathcal{X}_{\mathcal{H}}(t)\setminus\partial\mathcal{H}_0$  containing at least two points.
  By Lemma \ref{5b}, the internal ray $R_{U_\lambda}(t)$ lands at $v_\lambda^+$ for $\lambda\in \mathcal{E}$.
One may verify that for any $\lambda\in \mathcal{E}$, we have $f_\lambda^{k-2}(v_\lambda^+)\notin \partial B_\lambda$ and $f_\lambda^{k-1}(v_\lambda^+)\in \partial B_\lambda$. There is  disk neighborhood $\mathcal{D}\subset \mathbb{C}^*\setminus \partial \mathcal{H}_0$ of $\mathcal{E}$ such that for all $\lambda\in \mathcal{D}$, $f_\lambda^{k-2}(v_\lambda^+)\notin  \overline{B_\lambda}$.

Take two different parameters $\lambda_1,\lambda_2\in \mathcal{E}$ with $|\arg \lambda_1-\arg \lambda_2|<\frac{2\pi}{n-1}$ and
let $J=\{f_{\lambda_1}^j(v_{\lambda_1}^\varepsilon); 0\leq j \leq k-2, \varepsilon=\pm\}\cup \overline{B_{\lambda_1}}$.
We define a continuous map $h: \mathcal{D}\times J\rightarrow \mathbb{\widehat{C}}$ in the following way:

 1.  $h(\lambda_1, z)=z$ for all $z\in J$;

 2.  $h(\lambda, z)=\phi_\lambda^{-1}\circ \phi_{\lambda_1}(z)$ for all $(\lambda,z)\in \mathcal{D}\times \overline{B_{\lambda_1}}$;

 3.  For any $\lambda\in \mathcal{D}$,  we define $h(\lambda, f_{\lambda_1}^j(v_{\lambda_1}^\varepsilon))=f_{\lambda}^j(v_{\lambda}^\varepsilon)$ for $0\leq j\leq k-2$ and
  $\varepsilon\in \{\pm\}$.

   The map $h$ is a holomorphic motion parameterized by $\mathcal{D}$, with base point $\lambda_1$. By Slodkowski's theorem \cite{Sl}, there is a holomorphic motion $H:\mathcal{D}\times
\mathbb{\widehat{C}}\rightarrow\mathbb{\widehat{C}}$ extending $h$. We consider the restriction $H_0=H|_{\mathcal{E}\times
\mathbb{\widehat{C}}}$ of $H$.
  Note that for any  $\lambda\in \mathcal{E}$, the map $H_0(\lambda,\cdot)$ preserves the postcritical relation. 
  So  there is
  unique continuous map  $H_1: \mathcal{E}\times \mathbb{\widehat{C}}\rightarrow \mathbb{\widehat{C}}$ such that $H_1(\lambda_1,\cdot)\equiv id$ and
 the following diagram commutes:

  $$
\xymatrix{& \mathbb{\widehat{C}}\ar[r]^{f_{\lambda_1}} \ar[d]_{H_1(\lambda,\cdot)}
& \mathbb{\widehat{C}}\ar[d]^{H_0(\lambda,\cdot)}\\
&\mathbb{\widehat{C}}\ar[r]_{f_{\lambda}} & \mathbb{\widehat{C}}
 }$$

 Set $\psi_0=H_0(\lambda_2, \cdot)$ and $\psi_1=H_1(\lambda_2, \cdot)$.
  Both  $\psi_0$ and $\psi_1$ are quasiconformal maps satisfying  $f_{\lambda_2}\circ \psi_{1}=\psi_0\circ f_{\lambda_1}$.
   One may verify that $\psi_0$ and $\psi_1$ are homotopic rel $P(f_{\lambda_1})\cup \overline{B_{\lambda_1}}$. To see this, note that
   $H_1(\lambda, \cdot)^{-1}\circ H_0(\lambda, \cdot)$ is homotopic to the identity map rel $P(f_{\lambda_1})\cup \overline{B_{\lambda_1}}$
   for all $\lambda\in \mathcal{E}$.


  Then there is a sequence of quasi-conformal maps $\psi_j$ such that

 (a).  $f_{\lambda_2}\circ \psi_{j+1}=\psi_j\circ f_{\lambda_1}$
for all $j\geq0$,

(b). $\psi_{j+1}$ and $\psi_{j}$ are homotopic rel $f_{\lambda_1}^{-j}(P(f_{\lambda_1})\cup \overline{B_{\lambda_1}})$.


The maps $\psi_j$ form a normal family since their dilations are
uniformly bounded above.  Let  $\psi_\infty$ be the limit map of  $\psi_j$.
It is holomorphic in the Fatou set $F(f_{\lambda_1})=\cup_k f_{\lambda_1}^{-k}(B_{\lambda_1})$ and  satisfies
$f_{\lambda_2}\circ \psi_{\infty}=\psi_\infty\circ f_{\lambda_1}$ in $F(f_{\lambda_1})$. By continuity, $f_{\lambda_2}\circ \psi_{\infty}=\psi_\infty\circ f_{\lambda_1}$ in $\mathbb{\widehat{C}}$.

By Theorem \ref{lm},
 the  Lebesgue measures of $J(f_{\lambda_1})$ and $J(f_{\lambda_2})$ are zero.
Thus  $\psi_\infty$ is a
M\"obius map. It takes the form $\psi_\infty(z)=az$ where
$a^{n-1}=1$ and $\lambda_2=a^2\lambda_1$.
The condition $|\arg \lambda_1-\arg \lambda_2|<\frac{2\pi}{n-1}$ implies
$\lambda_1=\lambda_2$.
But this is a contradiction.
 \end{proof}

\begin{pro}\label{5d} The boundary $\partial\mathcal{H}$  is locally connected.
\end{pro}
\begin{proof} It follows from Lemma \ref{5c} that for any $t$, the impression $\mathcal{X}_{\mathcal{H}}(t)$ is either a singleton or
 contained in $\partial\mathcal{H}_0$. In the latter case,   for any  $\lambda\in \mathcal{X}_{\mathcal{H}}(t)\cap\mathcal{F}$ which is not  a cusp,
 it follows from Lemma \ref{3d} that there is an external ray $R_{\lambda}(\alpha)$ landing at
 $v_{\lambda}^+$.

 We claim that  $nt=n^{k-1}\alpha {\   \rm mod}\  1$. If not, then by Lemma \ref{3kk}, there is a cut ray $\Omega_\lambda^\beta$ separating $R_\lambda(nt)$ and
 $R_\lambda(n^{k-1}\alpha)$.  By stability of cut rays, there exist a neighborhood $\mathcal{U}$ of $\lambda$ and $\varepsilon>0$ such that
 $Z_{nt,\varepsilon}^u$ and $Z_{n^{k-1}\alpha,\varepsilon}^u$  are contained in different components of $\mathbb{\overline{C}}-\Omega_u^\beta$
 for all $u\in \mathcal{U}\cap (\mathbb{\overline{C}}-\mathcal{H}_0)$, where
 $$Z_{t,\varepsilon}^u:=\phi_u^{-1}(\{re^{2\pi i \theta}; r>1, |\theta-t|<\varepsilon\}).$$
  Moreover, by shrinking $\mathcal{U}$ a little bit, we see that $f_{u}^{k-1}(v_u^+)\in Z_{n^{k-1}\alpha,\varepsilon}^u$ for all
 $u\in\mathcal{H}\cap\mathcal{U}$. Then there is a cut ray $\Omega_u^\eta\subset f_u^{1-k}(\Omega_u^\beta)$, separating $v_u^+$ and
 $\cup_{|\theta-t|<\varepsilon/n}R_{U_u}(\theta)$ for all $u\in \mathcal{H}\cap\mathcal{U}$.
 However, by the definition of $\mathcal{X}_{\mathcal{H}}(t)$, when $k$ is large so that $1/k<\varepsilon/n$, there is $\lambda_k
 \in \mathcal{U}\cap \Phi_{\mathcal{H}}^{-1}(\{re^{2\pi i \theta}; 1-1/k<r<1,
|\theta-t|<{1}/{k}\})$. So we have $v_{\lambda_k}^+\subset \cup_{|\theta-t|<1/k}R_{U_{\lambda_k}}(\theta)\subset\cup_{|\theta-t|<\varepsilon/n}R_{U_{\lambda_k}}(\theta)$.
But this is a contradiction. This completes the proof of the claim.

Thus each $\lambda\in \mathcal{X}_{\mathcal{H}}(t)$ is  either a cusp or contained in $\{\nu(\alpha);
nt=n^{k-1}\alpha\}$(a finite set). The connectivity of $\mathcal{X}_{\mathcal{H}}(t)$ implies that it is a singleton.
\end{proof}

\begin{thm}\label{5d} The boundary $\partial\mathcal{H}$  is a Jordan curve.
\end{thm}

\begin{proof} If not, then there exist a parameter $\lambda\in\partial\mathcal{H}$ with $0\leq \arg\lambda\leq \frac{\pi}{n-1}$ (by assumption of $\mathcal{H}$) and two different angles $t_1,t_2$ such that  $\mathcal{X}_{\mathcal{H}}(t_1)=\mathcal{X}_{\mathcal{H}}(t_2)=\{\lambda\}$.


By Lemma \ref{3kk}, there is a cut ray $\Omega_\lambda^\alpha$ separating the internal rays $R_{U_\lambda}(t_1)$ and $R_{U_\lambda}(t_2)$. Suppose that $v_\lambda^+$ and $R_{U_\lambda}(t_1)$ are contained in the same component
  of $\mathbb{\widehat{C}}-\Omega_\lambda^\alpha$. By the  stability of cut rays, there is a neighborhood $\mathcal{U}$ of $\lambda$ such that
for any $u\in \mathcal{U}\cap \mathcal{H}$,   the set $\{v_u^+\}\cup R_{U_u}(t_1)$ and the internal ray $R_{U_u}(t_2)$ are contained in different components
  of $\mathbb{\widehat{C}}-\Omega_u^\alpha$. But contradicts the assumption that $\mathcal{X}_{\mathcal{H}}(t_2)=\{\lambda\}$.
\end{proof}






\section{Hyperbolic components of renormalizable type}\label{babe}

In this section, we study the  hyperbolic components  of renormalizable type.

We begin with a definition. We say a McMullen map $f_\lambda$ is {\it
renormalizable} (resp. {\it
$*$-renormalizable}) at $c\in C_\lambda$ if there exist an
integer $p\geq 1$  and two disks $U$ and $V$ containing $c$, such that
$f_\lambda^p: U\rightarrow V$ (resp. $-f_\lambda^p: U\rightarrow V$) is a quadratic-like map whose Julia set
is connected.  The triple $(f_\lambda^p, U, V)$ (resp. $(-f_\lambda^p, U, V)$) is called the {\it renormalization}
 (resp. $*$-{\it renormalization}) of $f_\lambda$.

Let $\mathcal{B}$ be  a hyperbolic component  of renormalizable type. For any $\lambda\in \mathcal{B}$, the map $f_\lambda$ has an attracting cycle in $\mathbb{C}$, say $z_\lambda\mapsto f_\lambda(z_\lambda)\mapsto\cdots\mapsto f^p_\lambda(z_\lambda)=z_\lambda$, where $p$ is the period.
We may assume that the attracting cycle is suitably chosen and labeled so that $z_\lambda$ is holomorphic with respect to $\lambda\in\mathcal{B}$.

\begin{lem}[\cite{WQY}, Prop 5.4]\label{6att} If $\lambda\in \mathcal{B}$,
  then $f_\lambda$ is either renormalizable or $*$-renormalizable. Moreover,

1. If $f_\lambda$ is renormalizable and $n$ is odd, then $f_\lambda$
has exactly two attracting cycles  in $\mathbb{C}$.

2. If $f_\lambda$ is $*$-renormalizable and $n$ is odd, then $p$ is even, $f^{p/2}_\lambda(z_\lambda)=-z_\lambda$ and
$f_\lambda$ has exactly one attracting cycle in $\mathbb{C}$.

3. If $n$ is even, then
$f_\lambda$ has exactly one attracting cycle in $\mathbb{C}$ and there is a unique $c\in C_\lambda$, such that
 $f_\lambda$ is renormalizable at $c$.
 \end{lem}

The terminology  `hyperbolic component  of renormalizable type' comes from Lemma \ref{6att}

Let $\rho(\lambda)=(f_\lambda^p)'(z_\lambda)$ be the multiplier of the attracting cycle of $f_\lambda$ for $\lambda\in\mathcal{B}$.
Base on Lemma \ref{6att}, we set $(\epsilon,k)=(-1,p/2)$ if $n$ is odd and $f_\lambda$
is $*$-renormalizable, and  $(\epsilon,k)=(1,p)$ in the other cases.
We define a map $\kappa:\mathcal{B}\rightarrow \mathbb{D}$ by $\kappa(\lambda)=(\epsilon f_\lambda^{k})'(z_\lambda)$. Note that either $\rho=\kappa^2$ or
$\rho=\kappa$.

The main result of this section is:

\begin{thm}\label{6a} The  map $\kappa:\mathcal{B}\rightarrow \mathbb{D}$  is a conformal map.
 It can be extended continuously to a homeomorphism from $\overline{\mathcal{B}}$ to $\overline{\mathbb{D}}$.
\end{thm}

\begin{proof} Note that $\kappa(\lambda)$ is the multiplier of the map $g_\lambda=\epsilon f_\lambda^{k}$ at its fixed point $z_\lambda$.
  By implicit function theorem, if  $\mathcal{B}\ni\lambda_n\rightarrow \partial\mathcal{B}$, then
$|\kappa(\lambda_n)|\rightarrow1$, so  the  map  $\kappa:\mathcal{B}\rightarrow \mathbb{D}$  is proper.

In the following,
we will show that $\kappa$ is actually a covering map. To this end, we will construct a local inverse map of $\kappa$ by means of quasiconformal
surgery. The idea  is similar to the quadratic case \cite{CG}.

Fix $\lambda_0\in \mathcal{B}$ and set $\kappa_0=\kappa(\lambda_0)$. We may relabel $z_{\lambda_0}$ so that the immediate attracting basin
$A_0$ of $z_{\lambda_0}$ contains a critical point $c\in C_{\lambda_0}$. Note that $\epsilon f_\lambda^{k}(A_0)=A_0$ and there is a conformal map
$\phi:A_0\rightarrow \mathbb{D}$ such that $\phi(z_{\lambda_0})=0$ and the following diagram commutes:
   $$
\xymatrix{& {A_0}\ar[r]^{\epsilon f_{\lambda_1}^k} \ar[d]_{\phi}
& {A_0}\ar[d]^{\phi}\\
&\mathbb{D}\ar[r]_{B_{\kappa_0}} & \mathbb{D}
 }$$
where $B_\zeta$ is the Blaschke product defined by $B_\zeta(z)=z\frac{z+\zeta}{1+\overline{\zeta}z}$.
Obviously $z=0$ is an attracting fixed point of $B_{\kappa_0}$ with multiplier $B'_\zeta(0)=\zeta$. Then there is a neighborhood $\mathcal{U}$ of $\kappa_0$ and
 a continuous family of quasiregular maps $\widetilde{B}:\mathcal{U}\times \mathbb{D}\rightarrow\mathbb{D}$ such that   $\widetilde{B}(\kappa_0,\cdot)=B_{\kappa_0}(\cdot)$ and
  $\widetilde{B}(\zeta,z)=B_{\kappa_0}(z)$ for  $\varepsilon<|z|<1$ ($\varepsilon>0$ is a small number),
 $\widetilde{B}(\zeta,z)=B_{\zeta}(z)$ for  $|z|<\varepsilon/2$ and $\widetilde{B}(\zeta,\cdot)$ is quasi-regular elsewhere.

Then we get a continuous family $\{G_\zeta\}_{\zeta\in\mathcal{U}}$ of quasiregular maps:
\begin{equation*}
G_\zeta(z)=\begin{cases}
 (-1)^q (f^{k-1}_{\lambda_0}|_{f_{\lambda_0}(A_0)})^{-1} (\epsilon \phi^{-1} \widetilde{B}(\zeta, \phi(e^{-q\pi i/n}z))),\ &z\in e^{q\pi i/n} A_0, \ 0\leq q< 2n,\\
f_{\lambda_0}(z),\ &z\in\mathbb{\widehat{C}}\setminus\cup_{0\leq q< 2n}e^{q\pi i/n} A_0.
\end{cases}
\end{equation*}

We can construct a $G_\zeta$-invariant complex structure $\sigma_\zeta$  such  that

$\bullet$ $\sigma_{\kappa_0}$ is  the standard complex structure $\sigma$ on $\mathbb{\widehat{C}}$.

$\bullet$ $\sigma_\zeta$ is continuous with respect to $\zeta\in\mathcal{U}$.

$\bullet$ $\sigma_{\zeta}$ is invariant under the maps $z\mapsto e^{2\pi i/n}z$ and  $z\mapsto -z$.

$\bullet$ $\sigma_{\zeta}$ is the standard complex structure near the attracting cycle and outside $\cup_{k\geq0}f_{\lambda_0}^{-k}(\cup_{0\leq q< 2n}e^{q\pi i/n} A_0)$.

The Beltrami coefficient $\mu_\zeta$ of $\sigma_\zeta$ satisfies $\|\mu_\zeta\|<1$. By Measurable Riemann Mapping Theorem, there is a continuous family of
quasiconformal maps $\psi_\zeta$ fixing $0, \infty$ and normalized so that $\psi_\zeta'(\infty)=1$. The map $\psi_\zeta$ satisfies
$\psi_\zeta(e^{2\pi i/n}z)=e^{2\pi i/n}\psi_\zeta(z)$ and $\psi_\zeta(-z)=-\psi_\zeta(z)$.
Then $F_\zeta=\psi_\zeta\circ G_\zeta\circ \psi_\zeta^{-1}$ is a rational map of the form
 $z^{-n}(z^{2n}+\sum_{0\leq k< 2n}b_k(\zeta)z^k)$. The symmetry $F_\zeta(e^{2\pi i/n}z)=F_\zeta(z)$
   implies $F_\zeta(z)=z^n+b_0(\zeta)z^{-n}+b_n(\zeta)$. Since the two free critical values of $F_\zeta$ satisfies $\psi_\zeta(v_{\lambda_0}^+)+
   \psi_\zeta(v_{\lambda_0}^-)=0$,  we have $b_n(\zeta)=0$. So $F_\zeta=f_{b_0(\zeta)}$. The coefficient $b_0:\mathcal{U}\rightarrow \mathcal{B}$ is continuous with $\kappa (b_0(\zeta))=\zeta$. So $b_0$ is the local inverse of $\kappa$. This implies $\kappa$ is a covering map. Since $\mathbb{D}$ is simply connected,
    $\kappa$ is actually a conformal map.

The map $\kappa$ has a continuation to the boundary $\partial \mathcal{B}$. By the implicit function theorem, the boundary $\partial \mathcal{B}$
is an analytic curve except at $\kappa^{-1}(1)$. So $\partial \mathcal{B}$ is locally connected. Since for any $\lambda\in \partial \mathcal{B}$,
the multiplier $e^{2\pi i t}$ of the non-repelling cycle of $f_\lambda$ is uniquely determined by its angle $t\in \mathbb{S}$,
 the  boundary $\partial \mathcal{B}$
is a Jordan curve.
\end{proof}

\begin{rmk} \label{6b} By Theorem \ref{6a}, the multiplier map $\rho:\mathcal{B}\rightarrow \mathbb{D}$ is a double cover  if and
 only if $n$ is odd and $f_\lambda$ is $*$-renormalizable. 
For example, when $n=3$,  let $\mathcal{B}_+$ (resp.  $\mathcal{B}_-$) be  the cardioid  of the `largest' baby Mandelbrot set intersecting the positive (resp. negative) real axis, then

1.  $\rho:\mathcal{B}_+\rightarrow \mathbb{D}$ is a conformal map and   near the center $\frac{1}{8}$ of $\mathcal{B}_+$,
  $$\rho(\lambda)=24(\lambda-\frac{1}{8})+(216+156\sqrt{2})(\lambda-\frac{1}{8})^2+\mathcal{O}((\lambda-\frac{1}{8})^3).$$

2.  $\rho:\mathcal{B}_-\rightarrow \mathbb{D}$ is a double cover and   near the center $-\frac{1}{8}$ of $\mathcal{B}_-$,
  $$\rho(\lambda)=576(\lambda+\frac{1}{8})^2+\mathcal{O}((\lambda+\frac{1}{8})^3).$$
\end{rmk}

\section{Appendix: Lebesgue measure}\label{leb}

In this appendix, we shall prove Theorem \ref{lm} based on the Yoccoz puzzle theory.

We first recall the construction of Yoccoz puzzles in \cite{WQY}. Given a parameter $\lambda\in \mathcal{M}\cap \mathcal{F}$, we define a graph
${G}_{\lambda}(\theta_1,\cdots,\theta_N)$ by
$${G}_\lambda(\theta_1,\cdots,\theta_N)=
\partial B_\lambda^L \cup \Big((\mathbb{\widehat{C}}\setminus B_\lambda^L)\cap
\cup_{k\geq0}\big(\Omega_\lambda^{\tau^k(\theta_1)} \cup
\cdots\cup\Omega_\lambda^{\tau^k(\theta_N)}\big)\Big),$$
where $L>1$ and $\theta_1,\cdots,\theta_N\in\Theta$ are $\tau$-periodic angles. The angles $\theta_1,\cdots,\theta_N$ are chosen so that
the free critical orbit $\cup_{k\geq1}f_\lambda^k(C_\lambda)$ avoids the graph. The puzzle pieces
of depth $d\geq0$  are defined to be  all the connected
components of $f_\lambda^{-d}((\mathbb{\widehat{C}}\setminus B_\lambda^L)\setminus
{G}_\lambda(\theta_1,\cdots,\theta_N))$. For any point $z\in J(f_{\lambda})$ whose orbit avoids the graph,
the puzzle piece of depth $d$ containing $z$ is denoted by $P_d^\lambda(z)$.
We say the graph ${G}_\lambda(\theta_1,\cdots,\theta_N)$ is {\it admissible} if there exists a non-degenerate critical annulus $P_{d}^\lambda(c)\setminus \overline{P_{d+1}^\lambda(c)}$ (or $P_{0}^\lambda(f_{\lambda}^d(c))\setminus \overline{P_{1}^\lambda(f_{\lambda}^{d}(c))}$) for some $c\in C_\lambda$ and some $d\geq1$.

\begin{lem}[\cite{WQY}, Prop 4.1] \label{adm} Suppose $\lambda\in \mathcal{M}\cap \mathcal{F}$ and the map $f_\lambda$ is postcritically infinite, then
 there exists an admissible
 graph ${G}_\lambda(\theta_1,\cdots,\theta_N)$.
\end{lem}

 For $c\in C_\lambda$, the tableau
$T(c)$ is defined as the two-dimensional array
$(P_{d}^\lambda(f_\lambda^l(x)))_{d,l\geq0}$. We say $T(c)$ is {\it periodic} if there is an integer $p\geq 1$ such that
 $P_{d}^\lambda(f_\lambda^p(c))=P_{d}^\lambda(c)$ for all $d\geq0$.


\begin{lem}[\cite{WQY}, Lemma 5.2 and Propositions 7.2 and 7.3] \label{tableau} Suppose $\lambda\in \mathcal{M}\cap \mathcal{F}$ and the graph
${G}_{\lambda}(\theta_1,\cdots,\theta_N)$ is admissible.

1. If $T(c)$ is periodic for some $c\in C_\lambda$, then $f_\lambda$ is either renormalizable or $*$-renormalizable. Let $K$ be the small filled Julia set of this
($*$-)renormalization, then $K\cap \partial B_\lambda$ contains at most one point.

2. If none of $T(c)$ with $c\in C_\lambda$ is periodic, then  for any sequence of shrinking puzzle pieces $P_0^\lambda\supset P_1^\lambda
\supset P_2^\lambda\cdots$, the intersection $\cap_{d\geq 0}\overline{P_d^\lambda}$ is a singleton.
\end{lem}

Here is a remark for Lemma \ref{tableau}: if some $T(c)$ is periodic,  then there exist $\epsilon\in\{\pm1\}$, $d\geq0$ and $p\geq1$ such that $(\epsilon f_\lambda^p, P_{d+p}^\lambda(c),  P_{d}^\lambda(c))$ is the ($*$-)renormalization (see \cite{WQY} for  more details); if none of $T(c)$ is periodic, by carrying out the Yoccoz puzzle theory one step further, we have

\begin{thm}[Lebesgue measure] \label{lm2} Suppose that $\lambda\in \mathcal{M}\cap \mathcal{F}$ and the graph
${G}_{\lambda}(\theta_1,\cdots,\theta_N)$ is admissible. If none of $T(c)$ with $c\in C_\lambda$ is periodic, then  the Lebesgue measure of $J(f_\lambda)$
is zero.
\end{thm}

\noindent{\it Proof of Theorem \ref{lm} assuming Theorem \ref{lm2}.}
We assume $\lambda\in \mathcal{F}_0$ (note that when $\lambda$ is real and positive, the map $f_\lambda$ is postcritically finite).
 It's known that if $f_\lambda$ is postcritically finite, then the Lebesgue measure of
$J(f_\lambda)$ is zero. So we assume further $\lambda\in \mathcal{F}$ and $f_\lambda$ is postcritically infinite. By Lemma \ref{adm}, there is an
admissible graph
${G}_{\lambda}(\theta_1,\cdots,\theta_N)$. By the assumption $f_\lambda^{k}(v_\lambda^+)\in \partial B_\lambda$, none of $T(c)$ with $c\in C_\lambda$ is periodic. It follows from Theorem \ref{lm2} that  the Lebesgue measure of $J(f_\lambda)$
is zero.
\hfill $\Box$

In this section, we actually prove Theorem \ref{lm2}  following Lyubich \cite{L}.






For $k\geq 0$, let $\mathcal{P}_k$ be
the collection of all puzzle pieces of depth $k$.
We first show that $d_k=\max\{{\rm diam}(P);P\in
\mathcal{P}_k\}\rightarrow0$ as $k\rightarrow\infty$. To see this,
suppose that there exist $\varepsilon>0$ and a sequence of puzzle
pieces $P_{n_k}\in\mathcal{P}_{n_k}$ with $n_1<n_2<\cdots$ and ${\rm
diam}(P_{n_k})\geq\varepsilon$. There is
$P^*_{n_1}\in\mathcal{P}_{n_1}$ such that $I_1=\{n_k; P_{n_k}\subset
P^*_{n_1}\}$ is an infinite set. For $k>1$, we define $P^*_{n_k}$
and $I_k$ inductively as follows: $P^*_{n_{k-1}}\supset
P^*_{n_k}\in\mathcal{P}_{n_k}$ and the set $I_k=\{j\in I_{k-1};
P_j\subset P^*_{n_k}\}$ is an infinite set. Then $P^*_{n_k},k\geq1$
is a sequence of shrinking puzzle pieces with  ${\rm
diam}(P^*_{n_k})\geq\varepsilon$. This contradicts the fact that
$\bigcap_k \overline{P^*_{n_k}}$ consists of a single point (see
Lemma \ref{tableau}).

We define the Yoccoz $\tau$-function as follows.
We choose some
$c\in C_\lambda$. For each $d\geq1$, we define $\tau(d)$ to be the
biggest integer $k\in [0,d-1]$ such that the puzzle piece
$f_\lambda^{d-k}(P_d^\lambda(c))$ contains some critical point in
$C_\lambda$, we set $\tau(d)=-1$ if no such integer exists. Since
$f_\lambda(e^{\pi i/n}z)=-f_\lambda(z)$, by the symmetry of puzzle
pieces (namely, $P_d^\lambda(e^{\pi i/n}z)=e^{\pi i/n}P_d^\lambda(z)$ for all
$d\geq1$, see Lemma 4.1 in \cite{WQY}), we see that the Yoccoz
$\tau$-function is well-defined (independent of the choice of $c\in
C_\lambda$). Moreover, it satisfies $\tau(d+1)\leq \tau(d)+1$.

We say that the critical set $C_\lambda$ is {\it non-recurrent} if
$\tau(d)$ is uniformly bounded for all $d\geq1$; {\it recurrent} if
$\limsup \tau(d)=\infty$ (this definition is in fact consistent with the definition in Section \ref{cut}); {\it persistently recurrent} if $\liminf
\tau(d)=\infty$.

Let $U\subsetneq\mathbb{C}$ be a simply connected planar domain and
$z\in U$. The shape of $U$ about $z$ is defined by:
$${Shape}(U,z)=\sup_{x\in \partial U}|x-z|/\inf_{x\in \partial U}|x-z|.$$

\begin{lem} \label{area}
Let $U,V$ be two planar disks with $V\Subset U\neq \mathbb{C}$, $x\in V$. Suppose that $Shape(U,x)\leq C$, $Shape(V,x)\leq C$,
$ mod(U-\overline{V})\leq m$, then there is a constant
$\delta=\delta(C,m)\in (0,1)$, such that
$area(V)\geq\delta area(U).$
\end{lem}

The proof of Lemma \ref{area} is based on the Koebe distortion theorem.
We leave it to the reader as an exercise.

\begin{lem} \label{density} Let $f$ be a rational map with Julia set
$J(f)\neq\mathbb{\widehat{C}}$. Let $z\in J(f)$,   if
there exist a number $\epsilon>0$, a sequence of integers $0\leq n_1< n_2<\cdots$ and a constant $N>0$ such that

1.  For any $k\geq0$, the component $U_k(z)$ of
$f^{-n_k}(B(f^{n_k}(z),\epsilon))$ that contains $z$ is a disk.

2. ${\rm deg} (f^{n_k}|_{U_k(z)})\leq N$ for all $k\geq1$.


Then $z$ is not  a Lebesgue density point of $J(f)$.
\end{lem}

\begin{proof}
By passing to a subsequence, we assume $f^{n_k}(z)\rightarrow w\in
J(f)$ as $k\rightarrow\infty$. We may assume further $z,w\neq \infty$ by a suitable change of coordinate. Choose $\epsilon_0< \epsilon$, when
$k$ is large,  we have $f^{n_k}(z)\in B(w,\epsilon_0/2)\subset
B(w,\epsilon_0)\subset B(f^{n_k}(z),\epsilon)$. Let  $V_k(z)$ be the
component of $f^{-n_k}(B(w,\epsilon_0/2))$ that contains $z$. Then
$V_k(z)$ is a disk and ${\rm deg} (f^{n_k}|_{V_k(z)})\leq N$.
 By
shape distortion (see \cite{WQY}, Lemma 6.1), the shape of $V_k(z)$
about $z$ is bounded above by some constant depending on $N$. We then
show that $diam(V_k(z))\rightarrow0$ as $k\rightarrow\infty$. In
fact, if not, again by choosing a subsequence, we assume
$V_k(z)$ contains a round disk $B(z,\rho)$ for some $\rho>0$. Then
for any large $k$, the image $f^{n_k}(B(z,\rho))$ is contained in $
B(w,\epsilon_0/2)$. But this contradicts the fact that  $J(f)\subset f^{n_k}(B(z,\rho))$ for large $k$ (see \cite{M}).

 Since
$J(f)\neq\mathbb{\widehat{C}}$, there is a round disk
 $B(\zeta, r)\Subset
B(w,\epsilon_0/2)\cap F(f)$, here $F(f)$ is the Fatou set of $f$.
Take a component $D_k$ of $f^{-n_k}(B(\zeta, r))$ in $V_k(z)$ and
$p\in f^{-n_k}(\zeta)\cap D_k$, then
by shape distortion (see \cite{WQY}, Lemma 6.1), there is a constant $C>0$ such that
 $Shape(D_k,p)\leq C Shape(B(\zeta, r),\zeta)=C$
 $Shape(V_k(z),p)\leq C Shape(B(w,\epsilon_0/2),\zeta)\leq C\epsilon_0/r$.
 Moreover,
$mod(V_k(z)\setminus\overline{D_k})\leq
mod(B(w,\epsilon_0/2)\setminus \overline{B(\zeta, r)}).$ It follows from Lemma \ref{area} that
there is a constant $\delta$ with $area(D_k)\geq \delta
area(V_k(z))$. So $$area(J(f)\cap V_k(z))\leq area(V_k(z)-D_k)\leq
(1-\delta)area(V_k(z)).$$ This implies $z$ is not a Lebesgue density
point.
\end{proof}

\begin{pro}\label{nonper} If $C_\lambda$ is not persistently recurrent, then the
Lebesgue measure of the Julia set $J(f_\lambda)$ is zero.
\end{pro}

\begin{proof}
Let $P(f_\lambda)=\overline{\cup_{k\geq 1}f_\lambda^k(C_\lambda)}$.
It's known (\cite{Mc4}, Theorem 3.9)  that for almost all $z\in
J(f_\lambda)$, the spherical distance
$d_{\mathbb{\widehat{C}}}(f_\lambda^n(z),P(f_\lambda))\rightarrow 0$ as $n\rightarrow\infty$.

Suppose that the critical set $C_\lambda$ is recurrent but not
persistently recurrent, then there is a positive integer $L$ such
that the set $\{k;\tau(k)\leq L\}$ is infinite. The recurrence of
$C_\lambda$ implies that there is $d\geq L$ such that the annulus
$P_d^\lambda(c)\setminus \overline{P_{d+1}^\lambda(c)}$ for some (hence all) $c\in C_\lambda$
is non-degenerate. Since $\tau(k+1)\leq \tau(k)+1$, the set
$\Lambda=\{k;\tau(k)=d, \tau(k+1)=d+1\}$ is infinite. Moreover
$\{f_\lambda^{k-d}(c); k\in\Lambda\}\subset \cup_{\zeta\in
C_\lambda}P_{d+1}(\zeta)\Subset \cup_{\zeta\in
C_\lambda}P_{d}(\zeta)$. So any critical point $c\in C_\lambda$
satisfies the conditions in Lemma \ref{density}, thus it is not a
Lebesgue density point. We consider a point $z\in
J(f_\lambda)\setminus C_\lambda$ with $\lim
d_\mathbb{\widehat{C}}(f_\lambda^n(z),P(f_\lambda))=0$.  We may assume that the forward
orbit of $z$ does not meet the graph ${G}_\lambda(\theta_1,\cdots,\theta_N)$ (for else
$z$ is not a Lebesgue density point by Lemma \ref{density}). In that
case, for each $k\in \Lambda$, there
is $n_k>k$ and $c'\in C_\lambda$ such that
$f_\lambda^{n_k-k-1}(P_{n_k}^\lambda(z))=P_{k+1}^\lambda(c')$ and
$f_\lambda^j(P_{n_k}^\lambda(z)), 0\leq j< n_k-k-1$ meets no critical point.
One can easily verify that $z$ satisfies the conditions in Lemma
\ref{density}, and  is not a Lebesgue density point of
$J(f_\lambda)$.

 If the critical
set $C_\lambda$ is not recurrent, one can verify that each point
$z\in J(f_\lambda)$ satisfies the condition in Lemma \ref{density}.
Thus $J(f_\lambda)$ carries no Lebesgue density point. The proof is
similar to, but easier than the previous argument. We omit the
details.
\end{proof}

We say a holomorphic map $g:\mathbf{U}\rightarrow \mathbf{V}$ is a
{\it repelling system} if $\mathbf{U}\Subset\mathbf{V}$,  the
boundary $\partial \mathbf{U}$ avoids the critical orbit of $g$ and
both $\mathbf{U}$ and $\mathbf{V}$ consist of finitely many disk
components. The filled Julia set of $g$ is defined by
$K(g)=\bigcap_{k\geq1} g^{-k}(\mathbf{V})$, it can be an empty set.

\begin{figure}[h]
\begin{center}
\includegraphics[height=6cm]{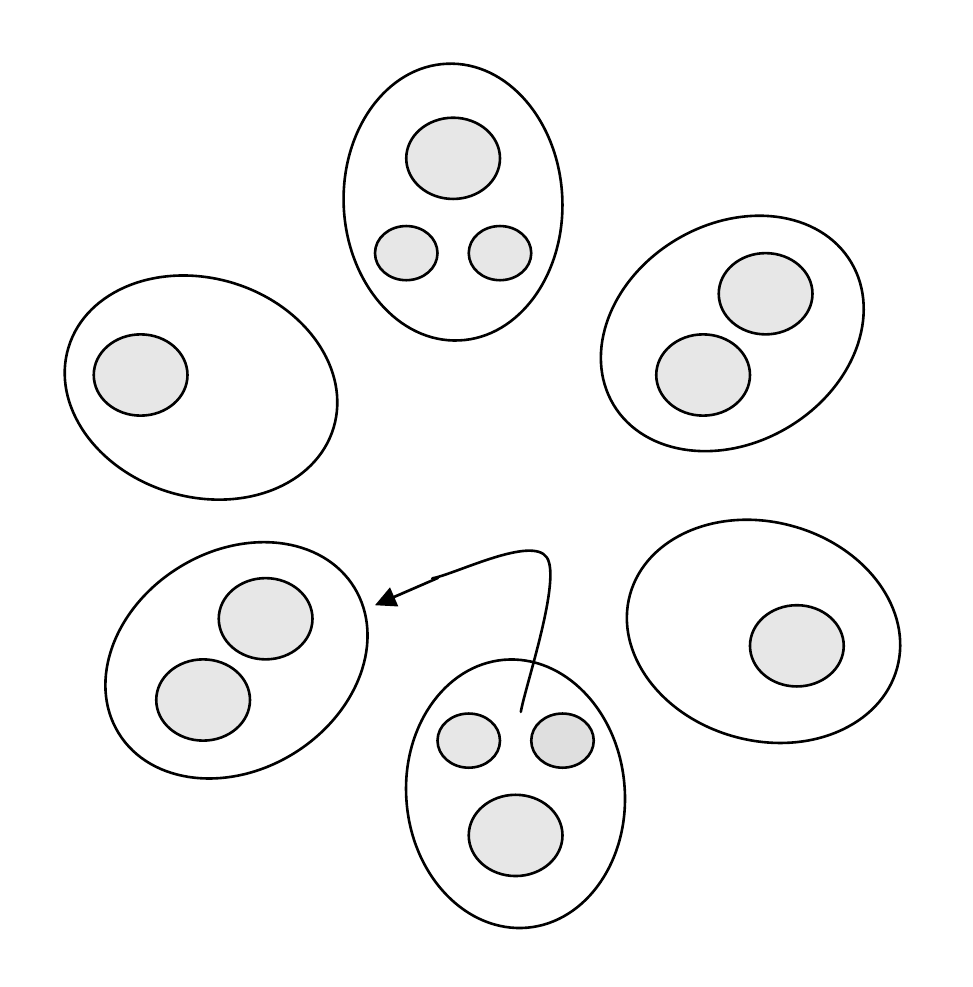}
\put(-80,80){$g$} \caption{A  repelling system
$g:\mathbf{U}\rightarrow \mathbf{V}$, where $\mathbf{U}$ is the union of
all shadow disks and  $\mathbf{V}$ is the union of six larger
disks.}
\end{center}\label{f5}
\end{figure}

\begin{thm}\label{pers} If the critical set $C_\lambda$  is persistently recurrent, then
there is a repelling system $g:\mathbf{U}\rightarrow \mathbf{V}$
such that

1. Each component of $\mathbf{U}$ and $\mathbf{V}$ is a puzzle
piece.

2. For each component $U_i$ of $\mathbf{U}$,
$g|_{U_i}=f_\lambda^{l_i}$ for some $l_i$.

3. $C_\lambda \subset K(g)$.

Moreover, the Lebesgue measure of
$J(f_\lambda)-\cup_{k\geq0}f_\lambda^{-k}(K(g))$ is zero.
\end{thm}

\begin{proof} Since the graph ${G}_\lambda(\theta_1,\cdots,\theta_N)$ is  admissible,
 we can find a non-degenerate critical annulus $P_d^\lambda(c)\setminus \overline{P^\lambda_{d+1}(c)}$ for some $d\geq 1$. Set $\mathbf{V}=\bigcup_{c\in
C_\lambda} P_{d+1}^\lambda(c)$. Then $f_\lambda^{j}(\partial \mathbf{V})\cap
\overline{\mathbf{V}}=\emptyset$ for all $j\geq 1$. For any $j\geq
1$, either $f_\lambda^{j}(C_\lambda)\subset \mathbf{V}$ or
$f_\lambda^{j}(C_\lambda)\cap \mathbf{V}=\emptyset$. Let $1\leq n_1<
n_2<\cdots$ be all the integers such that
$f_\lambda^{n_i}(C_\lambda)\subset \mathbf{V}$. Let
$l_i=n_{i+1}-n_{i} $ (set $n_0=0$) for $i\geq0$, we pull back
$\mathbf{V}$ along the orbit
$\{f_\lambda^{j}(C_\lambda)\}_{j=n_i}^{n_{i+1}}$ and get
$\mathbf{V}_i$. Namely, $\mathbf{V}_i$ is the union of all
components of $f_\lambda^{-l_i}(\mathbf{V})$ intersecting with
$f_\lambda^{n_i}(C_\lambda)$. For any $i$, the intermediate pieces
$f_\lambda^k(\mathbf{V}_i), 0<k<l_i$ lie outside $\mathbf{V}$ and
for any component  $V$ of $\mathbf{V}_i$, the map
$f_\lambda^{l_i}|_{V}$ is either univalent or a double covering.

Since $C_\lambda$ is persistently recurrent, the set $\{k;
\tau(k)\leq d+1\}$ is finite and  there are only finitely may
different $\mathbf{V}_i$'s. Moreover, if $\mathbf{V}_i\neq
\mathbf{V}_j$, then $\mathbf{V}_i\cap \mathbf{V}_j=\emptyset$ (In
fact, $\overline{\mathbf{V}}_i\cap
\overline{\mathbf{V}}_j=\emptyset$).

Let $\mathbf{U}=\bigcup_i \mathbf{V}_i$ and define
$g|_{\mathbf{V}_i}=f_\lambda^{l_i}$. Then $\mathbf{U}\Subset
\mathbf{V}$ follows from the fact that
$f_\lambda^{l_i}(\partial\mathbf{V}_i\cap
\partial\mathbf{V})\subset
\partial\mathbf{V} \cap
f_\lambda^{l_i}(\partial\mathbf{V})=\emptyset$.

It follows from  $\bigcup_{i\geq 0}
f_\lambda^{n_i}(C_\lambda)=\bigcup_{k\geq 0}g^k(C_\lambda)\subset
\mathbf{V}$ that $C_\lambda \subset K(g)$.

Similar to the proof of Proposition \ref{nonper}, we need only consider a point  $z\in
J(f_\lambda)$ with
$d_{\mathbb{\widehat{C}}}(f_\lambda^n(z),P(f_\lambda))\rightarrow 0$ as
$n\rightarrow\infty$.  For such point, there is an integer $N>0$ such that for
all $n\geq N$, $f_\lambda^n(z)\in \mathbf{V}$ implies
$f_\lambda^n(z)\in \mathbf{U}$. Note that there is  $p\geq N$ such
that $f_\lambda^{p}(z)\in \mathbf{V}$. Then for all $j\geq1$, we
have $g^j(f_\lambda^{p}(z))\in \mathbf{V}$. It turns out that
$f_\lambda^{p}(z)\in K(g)$. This implies
$J(f_\lambda)-\cup_{k\geq0}f_\lambda^{-k}(K(g))$ has zero Lebesgue
measure.
\end{proof}

Let $D\subset \mathbb{C}$ be a topological disk containing a compact
subset $K$ (not necessarily connected), the modulus of $A=D-K$,
denoted by $\mathbf{m}(A)$, is defined  to be the extremal length of
curves joining $\partial D$ and $\partial K$. It's equal to the
reciprocal of Dirichlet integral of
 the harmonic measure $u$ in $A$ (namely, $u$ is harmonic function in $A$ which tends to $0$ at regular points of $\partial K$
 and tends to $1$ at regular points of $\partial D$):
 $$\mathbf{m}(A)=\Big(\int_{A}|\nabla u|^2dxdy\Big)^{-1}.$$

If we require further that $K$ consists of finitely many components,
then we have the following area-modulus inequality (see \cite{L}):
$$area(D)\geq area(K)(1+4\pi \mathbf{m}(A)).$$

Now we  consider the repelling system $g:\mathbf{U}\rightarrow
\mathbf{V}$ defined in Theorem \ref{pers}. Set
$\mathbf{V}^0=\mathbf{V}$ and consider the preimages
$\mathbf{V}^d=g^{-d}(\mathbf{V})$ for $d\geq 1$. Note that
$\mathbf{V}^{d+1}\Subset \mathbf{V}^{d}$. For any $z\in K(g)$ and
$d\geq0$, denote by $\mathbf{V}^{d}(z)$  the piece of level $d$
containing $z$. Let
$\mathbf{A}^d(z)=\mathbf{V}^{d}(z)-\overline{\mathbf{V}^{d+1}}$, it
is a multiconnected domain. One can verify that for any $d\geq1$, if
$\mathbf{V}^{d}(z)$ contains no critical point in $C_\lambda$, then
$\mathbf{m}(\mathbf{A}^d(z))=\mathbf{m}(\mathbf{A}^{d-1}(f_\lambda(z)))$;
if $\mathbf{V}^{d}(z)$ contains a critical point in $C_\lambda$,
then
$2\mathbf{m}(\mathbf{A}^d(z))=\mathbf{m}(\mathbf{A}^{d-1}(f_\lambda(z)))$.

Using the same method as in \cite{L}, one can show that

\begin{lem}\label{Cantor} For any $z\in K(g)$, we have
$\sum_{d\geq1}\mathbf{m}(\mathbf{A}^d(z))=\infty$. It turns out that
$K(g)$ is a Cantor set.
\end{lem}

Now we have

\begin{thm}\label{zerom} Let $g:\mathbf{U}\rightarrow \mathbf{V}$
be the repelling system defined in Theorem \ref{pers}, then the
Lebesgue measure of $K(g)$ is zero.
\end{thm}
\begin{proof} For any $d\geq1$, let $\mathbf{V}^d(z_1),\cdots,
\mathbf{V}^d(z_{k_d})$ be all puzzle pieces of level $d$, where
$z_1, \cdots,z_{k_d}\in K(g)$. We define
$$M_d= \min_{1\leq i\leq k_d} \sum_{0\leq j<
d}\mathbf{m}(\mathbf{A}^j(z_i)).$$ By Lemma \ref{Cantor}, we have
$M_d\rightarrow \infty$ as $d\rightarrow\infty$. By area-modulus
inequality, we have
$$area(\mathbf{V}^{d}) \leq \frac{area(\mathbf{V})}
{\min_{1\leq i\leq k_d}\prod_{0\leq j< d}
(1+4\pi\mathbf{m}(\mathbf{A}^j(z_i)))} \leq \frac{area(\mathbf{V})}
{1+4\pi M_d}.$$ This implies $area(\mathbf{V}^{d})\rightarrow0$ as
$d\rightarrow\infty$.
\end{proof}

Theorem \ref{lm2} then follows from  Proposition \ref{nonper} and Theorems \ref{pers} and \ref{zerom}.

\end{document}